\documentclass[sn-mathphys,Numbered]{sn-jnl}

\usepackage{amsthm,amsmath,amsfonts,amssymb}
\usepackage{lmodern}
\usepackage{xcolor}
\usepackage{manyfoot}
\usepackage{booktabs}
\usepackage{subcaption}
\usepackage{algorithm}
\usepackage{algorithmicx}
\usepackage{algpseudocode}
\usepackage{listings}
\usepackage{nicematrix}
\usepackage{xurl}
\usepackage{extrastuff}
\usepackage{dirtytalk}

\newcommand\sd{H}

\newtheorem{theorem}{Theorem}
\newtheorem{corollary}{Corollary}
\newtheorem*{remark}{Remark}

\begin{document}

\title[Article Title]{Predictions Based on Pixel Data: Insights from PDEs and Finite Differences}

\author[1]{\fnm{Elena} \sur{Celledoni}}\email{elena.celledoni@ntnu.no}
\author[1]{\fnm{James} \sur{Jackaman}}\email{james.jackaman@ntnu.no}
\author[1]{\fnm{Davide} \sur{Murari}}\email{davide.murari@ntnu.no}
\author[1]{\fnm{Brynjulf} \sur{Owren}}\email{brynjulf.owren@ntnu.no}

\affil*[1]{\orgdiv{Department of Mathematical Sciences}, \orgname{Norwegian University of Science and Technology}, \orgaddress{\street{Alfred Getz' vei 1}, \city{Trondheim}, \postcode{7034}, \country{Norway}}}

\date{}

\abstract{As supported by abundant experimental evidence, neural networks are state-of-the-art for many approximation tasks in high-dimensional spaces. Still, there is a lack of a rigorous theoretical understanding of what they can approximate, at which cost, and at which accuracy. One network architecture of practical use, especially for approximation tasks involving images, is (residual) convolutional networks. However, due to the locality of the linear operators involved in these networks, their analysis is more complicated than that of fully connected neural networks. This paper deals with approximation of time sequences where each observation is a matrix. We show that with relatively small networks, we can represent exactly a class of numerical discretizations of PDEs based on the method of lines. We constructively derive these results by exploiting the connections between discrete convolution and finite difference operators. Our network architecture is inspired by those typically adopted in the approximation of time sequences. We support our theoretical results with numerical experiments simulating the linear advection, heat, and Fisher equations.}


\maketitle

\section{Introduction}

When endeavoring to understand the world around us, continuum modeling through partial differential equations {(PDEs)} \cite{Evans:1998,Strauss:2008} has proven a fundamental tool for a plethora of applications. Indeed, any process one may envision can be modeled by a differential equation through an appropriate simplification, be the focus fluid flows \cite[e.g.]{Vallis:2017}, mathematical biology \cite{Edelstein:2005}, chemical engineering \cite{Grzybowski:2009}, and most physical process. In practice, such models are typically solved using the rich toolbox established by numerical analysis and scientific computing. Popular techniques include finite difference methods \cite{Leveque:2007}, finite element methods \cite{BrennerScott:2007}, and spectral methods \cite{Boyd:2001}, and have proven excellent tools for simulating PDEs. However, there is a significant limitation, which is intrinsic to the inaccuracy of the PDE model itself. In other words, a PDE can only describe macroscopic behaviors that are sufficiently understood in terms of physical principles, and it can not capture phenomena at the molecular level exactly or incorporate information from observed data.

Over the past decade(s), there has been an explosion of research interest in artificial intelligence and machine learning, with applications in almost every area imaginable, like self-driving cars, chemistry, natural language processing, medical imaging, robotics, or weather forecasting \cite{Scher:2018,Bihlo:2021,li2014medical,jumper2021highly,chishti2018self,andrychowicz2020learning,brown2020language}. One increasingly popular application for machine learning is in the approximation of PDEs. Indeed, this has led to the development of several new technologies, such as physics-informed neural networks {(PINNs)} \cite{Raissi:2019,Cuomo:2022}. Such models are powerful tools for approximating PDE solutions. They typically build some PDE structure into the network and can be used both to solve PDEs and to learn PDEs from data. Besides PINNs, significant progress has been made in solving the inverse problem of discovering the underlying PDE \cite{rudy2019data, bar2019learning, xu2019dl}. Additional connections between PDEs and neural networks appear in the design of neural networks inspired by discrete numerical PDEs \cite{ruthotto2020deep,smets2023pde,eliasof2021pde}.

Our primary goal in this paper is to understand how accurately two-layer convolutional neural networks (CNNs) can approximate space-time discretizations of PDEs. An essential step towards incorporating physical knowledge into more standard networks is to construct neural networks capturing the structure of families of PDEs while being trained on data that is only an approximation of their solutions. In \cite{long2019pde}, the authors develop a methodology for expressing finite difference approximations with convolutional layers. This connection between discrete convolution and finite differences has also been used in \cite{szeliski2022computer,pratt2007digital,dong2017image} and is fundamental to the results of the present paper. The main difference between our work and the results presented in \cite{long2019pde} is the purpose of the study. In \cite{long2019pde}, the authors aim to discover a PDE given some snapshots of its solutions. They constrain their convolutional neural network architecture so that it can be seen as the semi-discretization of a PDE. Instead, in our work we study how deep convolutional filters have to be to represent the spatial semi-discretization of certain classes of PDEs. We do not impose restrictions on the entries of the filters, but we require specific choices of activation functions. See section \ref{se:error} for our analysis. We prove that for linear PDEs, a two-layer CNN with $\mathrm{ReLU}$ activation function and two channels can provide a second-order accurate semi-discretization of the PDE. A similar result holds for nonlinear PDEs with quadratic interaction terms.

In this paper, we consider space-time discretizations of PDEs. We restrict to two-dimensional spatial domains and take discrete snapshots in time where each observation is a matrix. While extending to higher dimensions and to tensors is not difficult, it is omitted here for the ease of notation. Sequences of matrices or tensors can be viewed as videos, and the task of learning the map from one snapshot to the next is an instance of the more general problem of next-frame prediction. In this context, unrestricted CNNs have been proven efficient \cite{mathieu2015deep,oprea2020review,fotiadis2020comparing,wang2022predrnn}. Similarly, in \cite{bertalan2019learning,greydanus2019hamiltonian,allen2020lagnetvip,khan2022hamiltonian}, the authors consider videos of dynamical systems, such as a physical pendulum, and approximate the underlying temporal evolution with a neural network inspired by ODEs and PDEs.

{Reliable numerical methods for differential equations must be both accurate and stable. Stability is a known issue also for neural networks, which are inherently sensitive to input perturbations such as adversarial attacks \cite{advAttacks}. The investigation of techniques to reduce networks' sensitivity to changes in the inputs is an active area of research \cite[e.g.]{celledoni2023dynamical,sherry2024designing,meunier2022dynamical}. {We analyze two techniques to enhance network stability in the context of next-frame predictions, intended as its ability to make reliable predictions further in the future.} The strategy that we propose relies on incorporating physical knowledge of the problem, specifically by preserving the norm of the PDE's initial condition, in the spirit of geometric numerical integration \cite{haier2006geometric}. It is also feasible to work directly with other properties of the PDE (such as conserved quantities) \cite{jin2020sympnets,bajars2023locally,cohen2016group,celledoni2023dynamical,wang2020incorporating,wang2020towards,zhu2020deep,cardoso2023exactly}. Incorporating such properties into the network often results in improved stability as we will see in the numerical experiments for the linear advection equation in the case of norm preservation.}

We will test our method on several PDEs, including the Fisher nonlinear reaction-diffusion equation. We anticipate the results for this PDE in figure \ref{fig:visualResultsFisher}, with details found in subsection \ref{sec:fisher}.

\begin{figure}[ht!]
	\centering
	\begin{subfigure}{.24\textwidth}
		\centering
		\includegraphics[width=\textwidth]{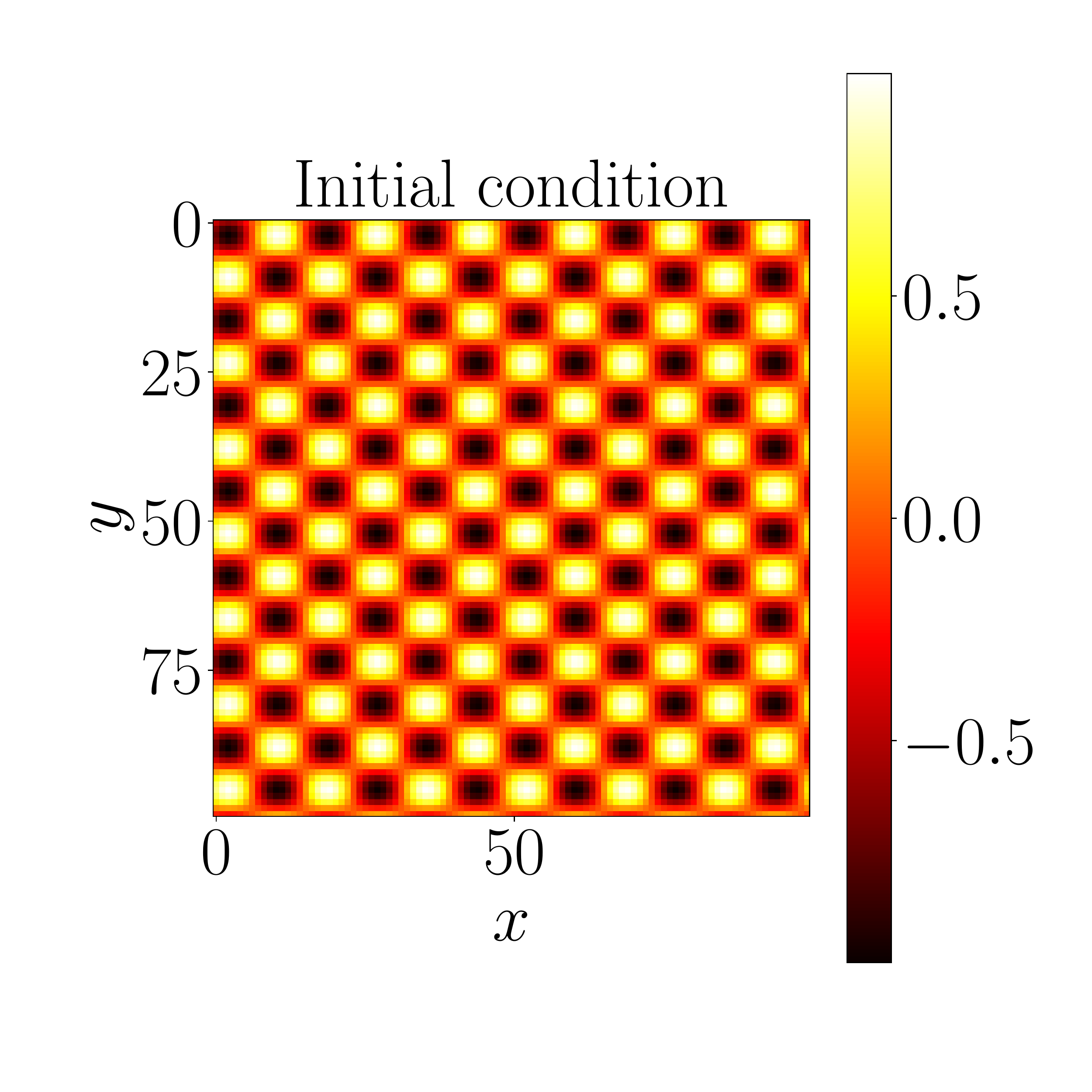}
	\end{subfigure}
	\begin{subfigure}{.24\textwidth}
		\centering
		\includegraphics[width=\textwidth]{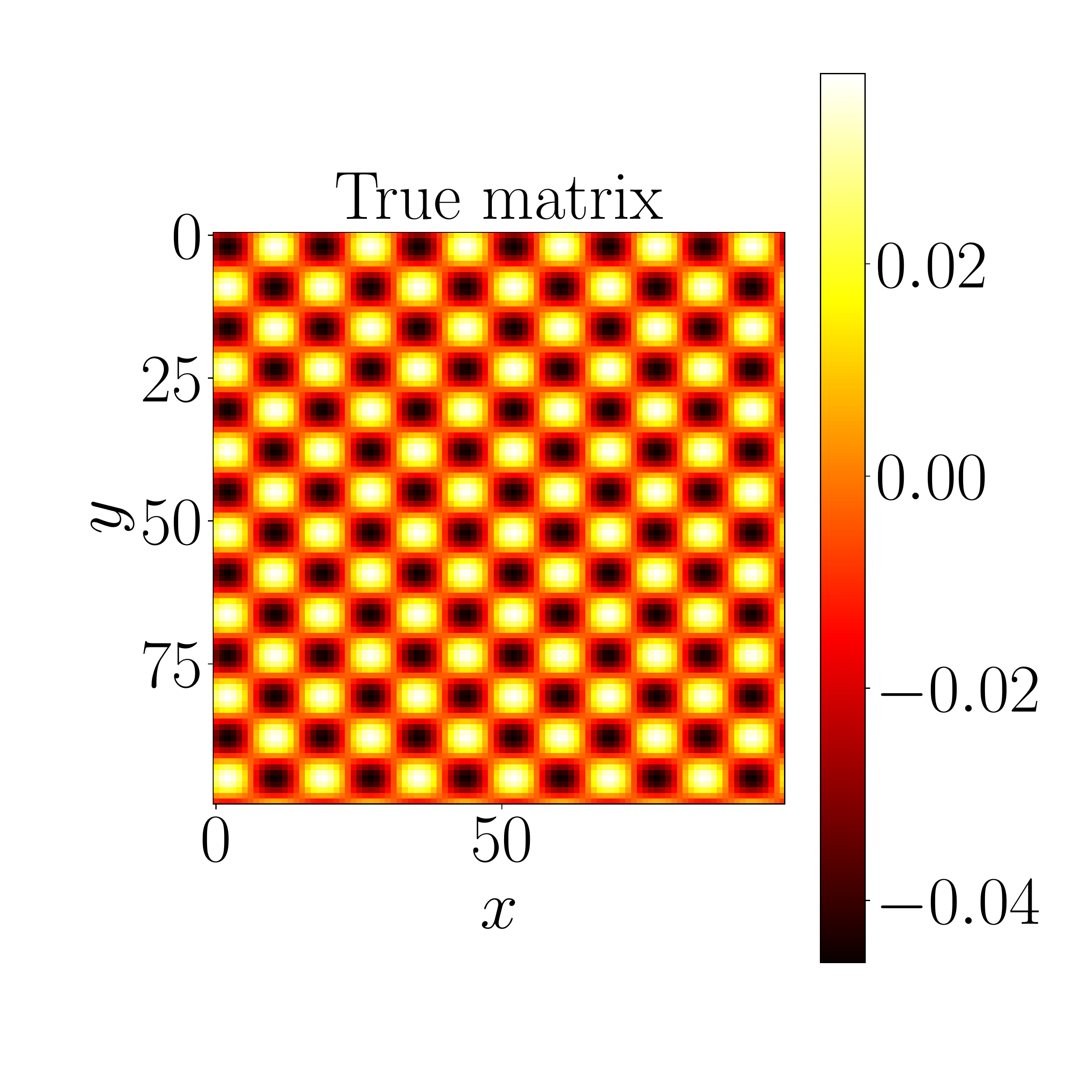}
	\end{subfigure}
	\begin{subfigure}{.24\textwidth}
		\centering
		\includegraphics[width=\textwidth]{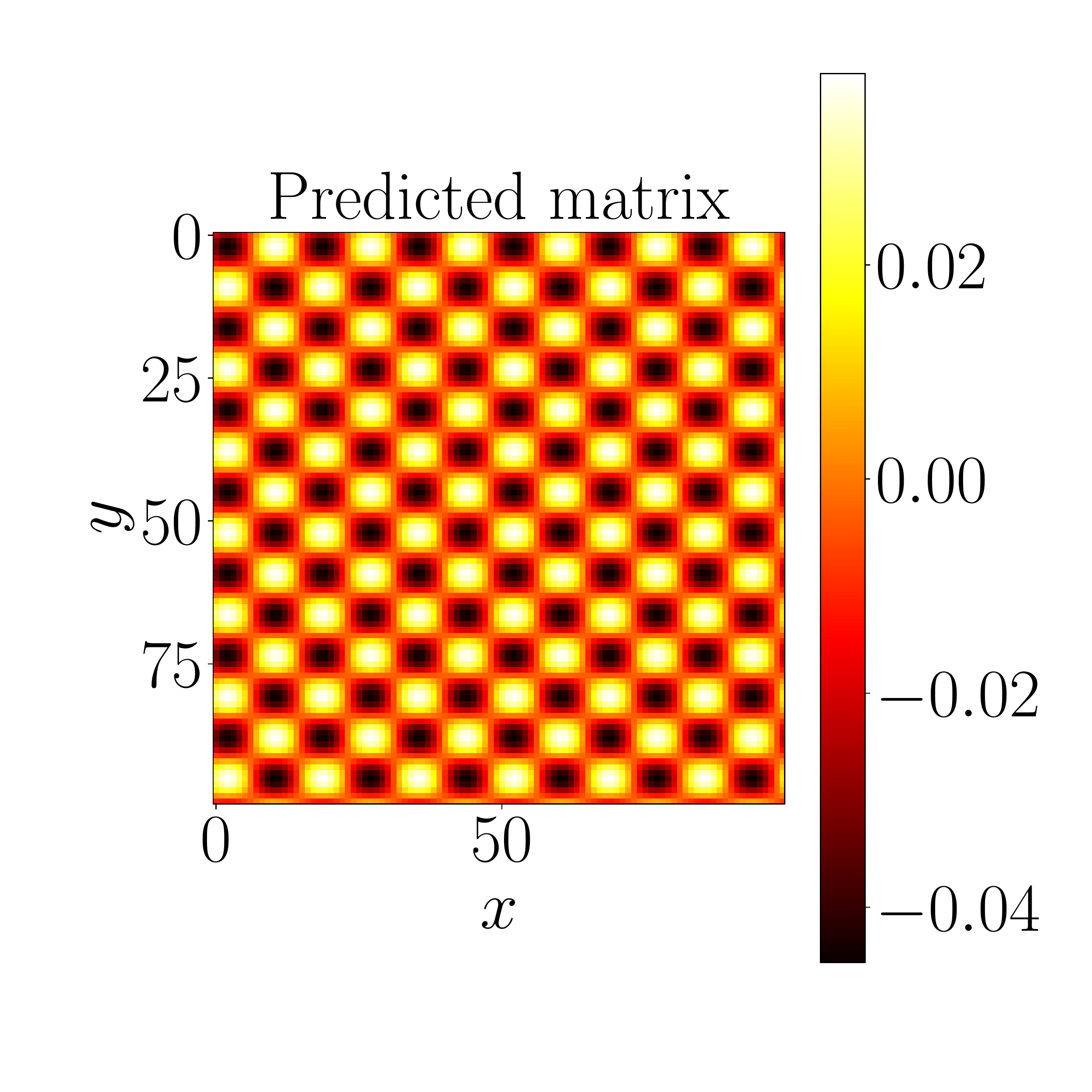}
	\end{subfigure}
	\begin{subfigure}{.24\textwidth}
		\centering
		\includegraphics[width=\textwidth]{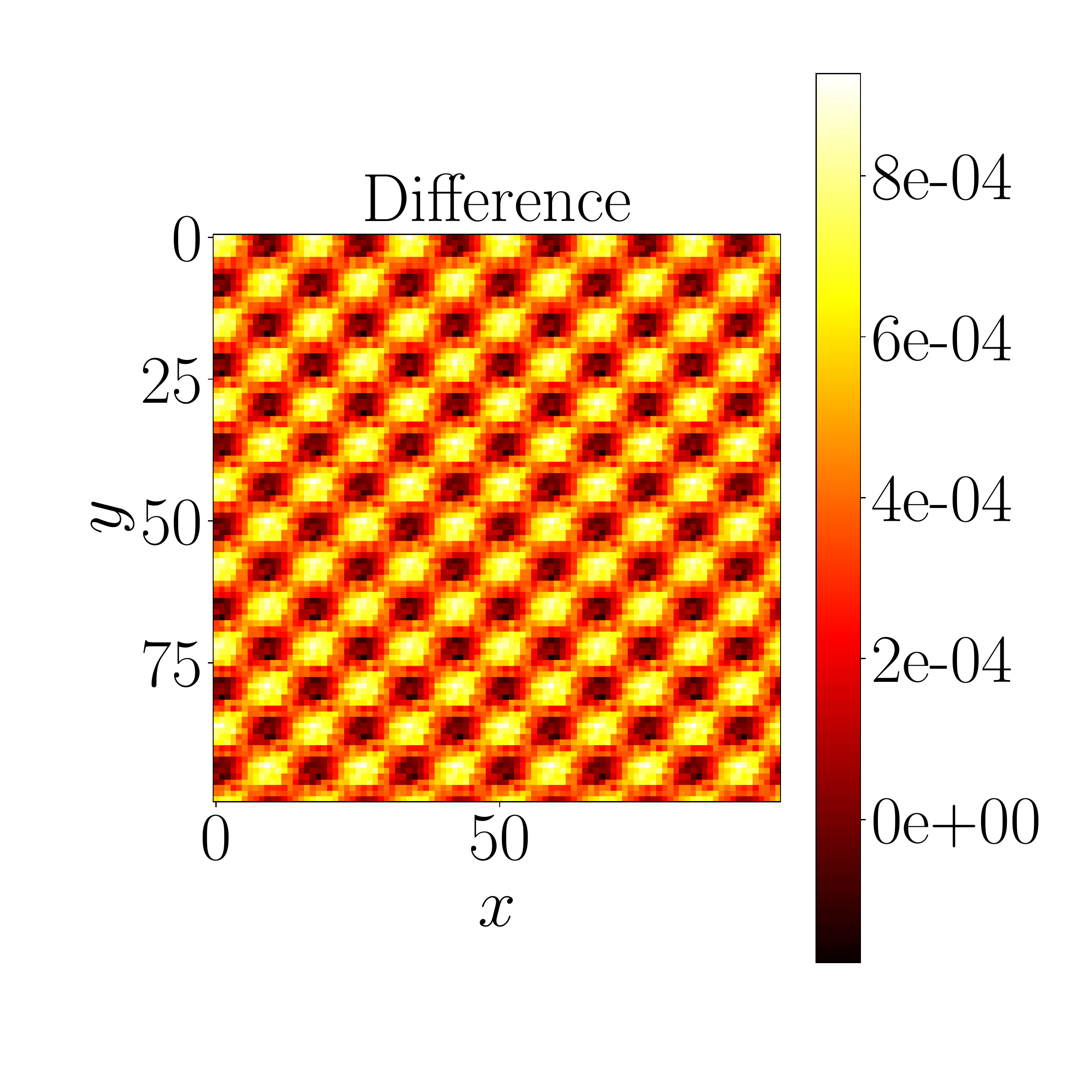}
	\end{subfigure}
	\caption{Visual representation of the prediction provided by the neural network. We consider one specific test initial condition, and snapshots coming from the space-time discretization of the reaction-diffusion equation \eqref{eq:fisher}. The network and PDE parameters can be found respectively in subsection \ref{sec:fisher} and appendix \ref{se:datagen}. From the left, we have the initial condition, the true space discretization of the solution at time $40\delta t$, the network prediction, and the difference between the last two matrices.}
	\label{fig:visualResultsFisher}
\end{figure}

The remainder of this paper is structured as follows: In section \ref{se:dynNetworks}, we discuss residual neural networks inspired by dynamical systems by introducing the fundamental tools used to build our method. {In section \ref{se:pdeSols} we provide an analysis of the error terms involved in the approximation of time-sequences generated by PDE discretizations. Section \ref{se:error}, specializes this analysis to PDEs on a two-dimensional spatial domain, presenting constructive approximation results}. Section \ref{se:stability} introduces two methodologies that we use to improve the stability of the learned map, i.e., injecting noise while training the model and preserving the norm of the solution when the underlying PDE is known to do so. In section \ref{se:numerical}, we perform numerical experiments verifying the good dynamical behavior of the proposed methods. In section \ref{se:conc}, we make concluding remarks and comment on future extensions of this research.
\section{Residual neural networks for time sequences}\label{se:dynNetworks}
We now present the considered problem for time sequences defined on a generic Euclidean space $\mathbb{R}^{\sd}$.

\textbf{Problem:} Let $\{(U_n^0,U_n^1,...,U_n^M)\}_{n=1}^N$ be a set of $N$ observed sequences each containing $M+1$ temporal snapshots represented by $U_n^m\in \mathbb{R}^{\sd}$. The superscript $m=0,...,M$ represents the observation time $t_m=m\,\delta t$, and the subscript $n=1,...,N$ refers to the considered initial condition $U_n^0\in \mathbb{R}^{\sd}$. Suppose there is a map $\Phi:\mathbb{R}^{\sd}\to \mathbb{R}^{\sd}$ such that
\begin{equation}\label{eq:problem}
	U_n^{m+1}=\Phi\left(U_n^m\right) \quad n=1,...,N,\,\,m=0,\cdots,M-1.
\end{equation}
We seek an accurate approximation of $\Phi$ that can reproduce unseen time sequences generated by the same dynamics (i.e., on a test set)\footnote{When we refer to a generic sequence, we suppress the subscript $n$ and write $\{U^m\}_{m=0,...,M}$.}. 

We consider sequences arising from numerical discretizations of scalar PDEs. Since a dynamical system generates our data, it is natural to approximate the map $\Phi$ with a ResNet, which results from composing $L$ functions of the form
\begin{equation}\label{eq:Presnet}
	x\mapsto x + F_{\theta_i}\left(x\right),\quad \theta_i\in\mathcal{P},\,i=1,...,L,
\end{equation}
where $\mathcal{P}$ is the space of admissible parameters, $F_{\theta_i}$ is a vector field on $\mathbb{R}^{\sd}$, $x\in \mathbb{R}^{\sd}$, and $L$ is the number of layers. One can see \eqref{eq:Presnet} as an explicit Euler step of size $1$ for the parametric vector field $x\mapsto F_{\theta_i}(x)$, \cite{weinan2017proposal,haber2017stable}.

Throughout, we denote with $\Phi^t_F$ the exact flow, at time $t$, of the vector field $F:\mathbb{R}^{\sd} \to \mathbb{R}^{\sd}$, and we use $\Psi^{\dt}_F$ to denote one of its numerical approximations at time $\dt$, for example $\Psi^{\dt}_F(x)=x+\dt \,F(x)$, $\dt =1$, in \eqref{eq:Presnet}.

We follow three steps for the approximation of $\Phi$:
\begin{enumerate}
	\item Define a family of vector fields $\{F_{\theta}:\,\theta\in\mathcal{P}\}$ capable of providing approximations of a spatial semi-discretization of the unknown PDE that is both qualitatively and quantitatively accurate.
	\item Choose a one-step numerical method $\Psi^{\delta t}_{F_{\theta}}$ that is compatible with the dynamics to be approximated.
	\item Choose a suitable loss function and minimize the discrepancy between the model predictions and the training data.
\end{enumerate}
\subsection{Characterization of the vector field}
We study parametrizations of $F_{\theta}$ of the form $F_{\theta}\left(U\right)=\mathcal{L}_2\sigma\left(\mathcal{L}_1\left(U\right)\right)$ where $\mathcal{L}_1,\mathcal{L}_2$ are suitable linear maps, and $\sigma$ is a non-linear activation function applied entry-wise. The linear maps we consider are represented by convolution operations with unrestricted filters. Our choice of the activation functions is so that $F_{\theta}$ can represent products between derivative discretizations that typically appear in PDEs. The details for this characterization of $F_{\theta}$ are provided in section \ref{se:error}. Our construction is inspired by \cite{alguacil2021effects,mathieu2015deep,long2019pde}. 
\subsection{Numerical time integrator}
Given a parametric vector field $F_{\theta}:\mathbb{R}^{\sd}\to \mathbb{R}^{\sd}$, we define a neural network $\mathcal{N}_{\theta}\left(U\right):= \Psi^{\dt}_{F_{\theta}}\left(U\right)$ by using the numerical method $\Psi^{\dt}_{F_{\theta}}$. In our experiments, the map $\Psi^{\dt}_{F_{\theta}}$ will generally be defined by composing $k$ substeps with a numerical method $\varphi^{\dt/k}_{F_{\theta}}$ of time step $\dt/k$.  Thus, we can express $\mathcal{N}_{\theta}$ as 
\[
\mathcal{N}_{\theta}\left(U\right) = \Psi^{\dt}_{F_{\theta}}\left(U\right) = \underbrace{\varphi^{\dt/k}_{F_{\theta}}\circ ... \circ \varphi^{\dt/k}_{F_{\theta}}}_{k\text{ times}}\left(U\right).
\]
This can be seen as a $k$-layer neural network with shared weights. In most of our experiments, $\varphi_{F_{\theta}}^{\dt/k}$ corresponds to the explicit Euler method. The chosen numerical method can also be structure-preserving in case the considered dataset has some property that is worth preserving. Experimentally, we shall consider one such problem in the linear advection equation, which is norm preserving. We will see that preserving such an invariant leads to improved stability in the predictions. We compare this to noise injection \cite{bishop1995training}, see section \ref{se:stability}.

\subsection{Optimization problem to solve}
We conclude this section with the third step, reporting the loss function optimized to find the final approximate function $\mathcal{N}_{\theta}$. Since this optimization step is not the main focus of the paper, we adopt the standard mean squared error loss function defined for the full dataset as 
\begin{equation}\label{eq:lossPixel}
	\mathcal{L}\left(\theta,Q\right) = \frac{1}{N\cdot Q} \sum_{n=1}^N \sum_{q=1}^Q \left\|\mathcal{N}_{\theta}^q\left(U_n^0\right) - \Phi^q\left(U_n^0\right)\right\|^2,
\end{equation}
where $Q<M$ is the number of steps we perform while training, and $\mathcal{N}_{\theta}^q = \underbrace{\mathcal{N}_{\theta}\circ \cdots  \circ \mathcal{N}_{\theta}}_{q\text{ steps}}$. We note that $\Phi^q(U_n^0)=U_n^q$, as seen in \eqref{eq:problem}. We use the Adam optimizer to minimize the loss in our experiments. 

The strategy of composing $q$ times the neural network while optimizing its weights helps reduce the accumulation of error due to the iterative application of the neural network, see \cite{chen2019symplectic,celledoni2023learning}. {This strategy, hence, positively affects the network's temporal stability.}

\section{Error bounds for network-based approximations of PDE solutions}\label{se:pdeSols}
Let us restrict our focus to seeking solutions $u:\mathbb{R}\times \Omega\to \mathbb{R}$, $\Omega\subset\mathbb{R}^d$, which satisfy PDEs of the form
\begin{equation}\label{eq:pde}
	\partial_t u = \mathcal{L}u + \sum_{i=1}^I\beta_i\,\mathcal{D}_i^a u\cdot\mathcal{D}_i^b u,\quad \beta_i\in\mathbb{R},\,\,t\geq 0,\,\,x\in\Omega,
\end{equation}
where $\mathcal{L},\mathcal{D}_i^a,\mathcal{D}_i^b$ are linear differential operators in the spatial variable $x\in \mathbb{R}^d$, and $I$ is the number of quadratic interactions in the PDE.

A commonly adopted strategy to find approximate solutions to \eqref{eq:pde} is to use the method of lines (see, e.g., \cite{schiesser2009compendium,schiesser2012numerical}). We introduce a spatial discretization of the differential operators based on $\sd$ spatial nodes in $\Omega$ and obtain an ODE of the form
\begin{equation}\label{eq:semi}
	\dot{U}\left(t\right) = LU\left(t\right) + \sum_{i=1}^I\beta_i \left(D_{i}^aU\left(t\right)\right)\odot \left(D_{i}^bU\left(t\right)\right) = F\left(U\left(t\right)\right)\in \mathbb{R}^{\sd},
\end{equation}
where the components of $U$ are approximations of $u$ in the grid-points of the spatial discretization, $\odot$ is the entry-wise product, and $LU$, $D_i^aU$, and $D_i^bU$ are spatial discretizations of $\mathcal{L}u$, $\mathcal{D}_i^au$, and $\mathcal{D}_i^bu$. Going forward, we assume that the dataset $\{(U^0,U^1,...,U^M)\}$, where $U^m\in \mathbb{R}^{\sd}$, provides an approximate solution to \eqref{eq:semi}. That is to say that by introducing a time step $\delta t$ and defining $t_m=m\delta t$, we have $U^m\approx U(t_m)$ for $t\mapsto U(t)\in\mathbb{R}^{\sd}$ which is the exact solution of \eqref{eq:semi}.

Let $x_h$, $h=1,...,{\sd}$, be a generic point on the spatial grid over $\Omega$. Then, $(U^m)_h\approx u(t_m,x_h)$ where $\mathbb{R}\times \Omega\ni (t,x)\mapsto u(t,x)\in\mathbb{R}$ is the analytical solution of \eqref{eq:pde}. To account for possible measurement errors, we introduce the function $e: \mathbb{R}\times \Omega\to\mathbb{R}$ representing the difference between the data and the true solution $u$, which allows to write
\begin{align}
	\left(U^0\right)_{h} &= u\left(0,x_h\right) \\
	\left(U^m\right)_{h} &= u\left(t_m,x_h\right) + e\left(t_m,x_h\right),\quad m=1,...,M.\label{eq:dataDescription}
\end{align}

\subsection{Splitting of the approximation errors}
Our goal is to find a map $\mathcal{N}_{\theta}$ which accurately approximates the function
\[
U^m \mapsto \Phi\left(U^m\right) = U^{m+1},\quad m=0,...,M-1.
\]
We focus on the case $m=0$, but the same analysis can be done for the other values of $m$. Let $\Phi^{\dt}_F$ be the exact flow map of the ODE in \eqref{eq:semi}. Considering the spatial error due to the discretization, which we assume is of a generic order $k$, \eqref{eq:dataDescription} implies
\[
U^{1} = \Phi^{\dt}_F\left(U^0\right) + \mathcal{O}\left(\dx^k\right) + \varepsilon^{1},
\]
where $\varepsilon^{1}\in \mathbb{R}^{\sd}$ is defined such that $\left(\varepsilon^1\right)_h=e\left(t_1,x_h\right)$, $\delta x$ quantifies the mesh size of the spatial grid, and $k$ is the order of the method used to spatially discretize the PDE \eqref{eq:pde} to obtain \eqref{eq:semi}. We focus on approximating the unknown ODE \eqref{eq:semi}. More precisely, we will evaluate how well we can approximate the flow map $\Phi^{\dt}_F$ of $F$ given a parametric set of functions 
\begin{equation}\label{eq:2layernet}
	\mathcal{F}=\left\{F_{\theta}\left(U\right)=\mathcal{L}_2\left(\theta\right)\sigma\left(\mathcal{L}_1\left(\theta\right)\left(U\right)\right)\in\mathbb{R}^{\sd}:\,\,\theta\in \mathcal{P}\right\}
\end{equation}
for a set of admissible parameters $\mathcal{P}$ inducing mappings of the form
\[
U\mapsto \Psi^{\dt}_{F_{\theta}}\left(U\right),\,\,F_{\theta}\in\mathcal{F}.
\]
For simplicity, we will suppress the dependency of $\mathcal{L}_1$ and $\mathcal{L}_2$ on $\theta$ writing $F_{\theta}(U)=\mathcal{L}_2\sigma(\mathcal{L}_1U)$ for two linear maps $\mathcal{L}_1=\mathcal{L}_1(\theta)$ and $\mathcal{L}_2=\mathcal{L}_2(\theta)$.
Through splitting the local error, starting from position $U^0$, we may write
\begin{equation}
	\begin{split}
		&\left\|U^{1} - \Psi^{\dt}_{F_{\theta}}\left(U^{0}\right)\right\| \\
		&=\left\|\varepsilon^{1} + \mathcal{O}\left(\dx^k\right) + \Phi^{\dt}_F\left(U^0\right) - \Psi^{\dt}_{F_{\theta}}\left(U^{0}\right)\right\| \\
		&\leq \left\|\ \varepsilon^{1}\right\|\ + \mathcal{O}\left(\dx^k\right) +  \left\|\Phi^{\dt}_F\left(U^0\right) - \Psi^{\dt}_{F}\left(U^0\right) + \Psi^{\dt}_{F}\left(U^0\right) - \Psi^{\dt}_{F_{\theta}}\left(U^0\right)\right\|\\
		& \leq  \underbrace{\left\|\ \varepsilon^{1}\right\|\ }_{\text{measurement error}} + \underbrace{\mathcal{O}\left(\dx^k\right)}_{\text{spatial error}} + \underbrace{\left\|\Phi^{\dt}_F\left(U^0\right) - \Psi^{\dt}_{F}\left(U^0\right)\right\|}_{\text{classical error estimate}}\\
		&\qquad+\underbrace{\left\|\Psi^{\dt}_{F}\left(U^0\right) - \Psi^{\dt}_{F_{\theta}}\left(U^0\right)\right\|}_{\text{network approximation}}.
	\end{split}
	\label{eq:errorSplit}
\end{equation}
In this chain of inequalities, $\|\cdot\|$ denotes the Euclidean norm on $\mathbb{R}^{\sd}$. The estimate's classical error term only depends on the local truncation error of the numerical method $\Psi^{\delta t}$. Indeed, this term is $\mathcal{O}(\dt^{r+1})$ if $\Psi^{\dt}$ is a method of order $r$, allowing us to write $\Psi^{\dt}_F\left(U^0\right) = \Phi^{\dt}_F\left(U^0\right) + \mathcal{O}\left(\dt^{r+1}\right)$, which leads to
\[
\left\|U^{1} - \Psi^{\dt}_{F_{\theta}}\left(U^{0}\right)\right\| \leq \mathcal{O}\left(\dx^k\right) + \mathcal{O}\left(\dt^{r+1}\right) + \left\|\varepsilon^{1}\right\| + \left\|\Phi^{\dt}_{F}\left(U^{0}\right) - \Phi^{\dt}_{F_{\theta}}\left(U^{0}\right)\right\|.
\]
We assume that for all the considered initial conditions $U_n^0$ and time instants $t\in [0,\delta t]$, the vectors $\Phi_F^{t}(U_n^0)$ and $\Phi_{F_{\theta}}^t(U_n^0)$ belong to a compact set $\Omega\subset\mathbb{R}^{\sd}$. Restricting to $\Omega$, the vector field $F$ in \eqref{eq:semi} is Lipschitz continuous, with a Lipschitz constant denoted as $\mathrm{Lip}(F)$. To handle the second term on the right-hand side of \eqref{eq:errorSplit}, we work with Gronwall's inequality applied to the integral representation of the flow map:
\begin{equation}\label{eq:timeBound}
	\begin{split}
		&\left\|\Phi^{\dt}_F\left(U^0\right)-\Phi^{\dt}_{F_{\theta}}\left(U^0\right)\right\| \leq \int_0^{\dt} \left\| F\left(\Phi_F^s\left(U^0\right)\right)-F_{\theta}\left(\Phi_{F_{\theta}}^s\left(U^0\right)\right) \right\| \di{s}\\
		&= \int_0^{\dt} \left\| F\left(\Phi_F^s\left(U^0\right)\right)- F\left(\Phi_{F_{\theta}}^s\left(U^0\right)\right)+ F\left(\Phi_{F_{\theta}}^s\left(U^0\right)\right) - F_{\theta}\left(\Phi_{F_{\theta}}^s\left(U^0\right)\right) \right\|\,\di{s} \\
		&\leq \mathrm{Lip}\left(F\right)\int_0^{\dt}\left\|\Phi_F^s\left(U^0\right)-\Phi^s_{F_{\theta}}\left(U^0\right)\right\|\,\di{s} + \dt\sup_{V\in\Omega}\left\|F_{\theta}\left(V\right)-F\left(V\right)\right\|\\
		&\implies \left\|\Phi^{\dt}_F\left(U^0\right)-\Phi^{\dt}_{F_{\theta}}\left(U^0\right)\right\| \leq \dt\exp\left(\mathrm{Lip}\left(F\right)\dt\right)\sup_{V\in\Omega}\left\|F_{\theta}\left(V\right)-F\left(V\right)\right\|.
	\end{split}
\end{equation}
Thus, to control this quantity from above, we need to understand the approximation properties of $\mathcal{F}$. In particular, we want to quantify how much complexity $\mathcal{F}$ requires to guarantee the existence of an $F_{\theta}\in\mathcal{F}$ leading to
\begin{equation}\label{eq:gronwall}
	\sup_{V\in\Omega}\left\|F_{\theta}\left(V\right)-F\left(V\right)\right\| < \dt^q
\end{equation}
for some $q\geq 1$. We note that if $q=r$, such a result would guarantee that the approximation of the map $\Phi$ provided by $\Psi^{\dt}_{F_{\theta}}$ can be as accurate as the one provided by the numerical method $\Psi^{\dt}_F$ directly applied to the exact vector field $F$. We now characterize the space $\mathcal{F}$ so that it can exactly represent $F$, i.e., so that there exists an element $F_{\theta}\in\mathcal{F}$ with $F=F_{\theta}$. In practice, this exact representation will never be obtained, and the presented derivation allows the quantification of the approximation error. In the next section, we will focus on two-dimensional PDEs, i.e., $d=2$, and analyze \eqref{eq:gronwall} for that setting.

\section{Error analysis for PDEs on a two-dimensional spatial domain}\label{se:error}
We now focus on PDEs defined on a two-dimensional spatial domain. These PDEs are discretized on a uniform grid with $p$ grid points along both the axes over the spatial domain $\Omega=[0,1]^2\subset\mathbb{R}^2$, i.e., we set $d=2$ and $\sd=p^2$. In this two-dimensional setting, we denote the grid points by $(x_h,y_k)$, $h,k=1,...,p$. Since the grid is uniform, we have $x_{h+1}-x_h=y_{k+1}-y_k=:\delta x$. For convenience, we represent an element in $\mathbb{R}^{p^2}$ as a matrix in $\mathbb{R}^{p\times p}$. The only immediate consequence is that the previously derived estimates involving the Euclidean norm of $\mathbb{R}^{\sd}$ now hold in the Frobenius norm of $\mathbb{R}^{p\times p}$. We assume that the PDE is at most of the second order and that the spatial semi-discretization in \eqref{eq:semi} comes from a second-order accurate finite differences scheme, i.e., $k=2$ in \eqref{eq:errorSplit}. {By restricting ourselves to the two-dimensional case, second-order PDEs, and second-order accurate finite differences, we significantly simplify the exposition. Our reasoning may be extended to higher-dimensional domains discretized with regular grids and higher-order finite differences.} 

For our derivations, we use tensors of orders three and four. We denote a tensor of order four as
\begin{equation}\label{eq:fourthOrderTensor}
	\fb{\mathcal{X}} = \left[\boldsymbol{X}_1,...,\boldsymbol{X}_J\right]\in\mathbb{R}^{J\times K\times R\times S}
\end{equation}
where $\boldsymbol{X}_j\in\mathbb{R}^{K\times R\times S}$ is a tensor of order three for every $j=1,...,J$ defined as
\begin{equation}\label{eq:thirdOrderTensor}
	\boldsymbol{X}_j = \left[X_{j,1},...,X_{j,K}\right] \in\mathbb{R}^{K\times R \times S},
\end{equation}
where $X_{j,k}\in\mathbb{R}^{R\times S}$ for every $k=1,...,K$. We may access the components of fourth-order tensors using the notation $\fb{\mathcal{X}}_j = \boldsymbol{X}_j\in\mathbb{R}^{K\times R\times S}$, and $\fb{\mathcal{X}}_{j,k} = X_{j,k}\in\mathbb{R}^{R\times S}$, for $j=1,...,J$ and $k=1,...,K$. When $\boldsymbol{X}=[X_1]\in\mathbb{R}^{1\times R\times S}$, we will often contract the first dimensionality and refer to it as $X_1\in\mathbb{R}^{R\times S}$. 

We now show that any function with the same structure as $F$, as described in \eqref{eq:semi}, can be represented by the parametric space of functions $\mathcal{F}$ defined in \eqref{eq:2layernet} with linear maps realized by convolution operations.

\subsection{Convolutional layers as finite differences}\label{se:conv}
The discrete convolution operation is the foundation of many successful machine learning algorithms, particularly for approximation tasks involving images. This work focuses on \say{same} convolutions, i.e., convolution operations that do not change the input dimension. In this case, the input matrix has to be padded compatibly with the application of interest. Specifically, the padding strategy in our setting should relate to how the PDE solution $u$ behaves outside the domain $[0,1]^2$, i.e., to the boundary conditions. We focus on the case of periodic boundary conditions, as considered in the numerical experiments. 

For the specific situation $U\in\mathbb{R}^{3\times 3}$, i.e., $p=3$, and $3\times 3$ convolutional filters, the periodic padding leads to
\[
U_P = 
\begin{bNiceMatrix}[margin]
	\CodeBefore
	\columncolor{red!15}{1,5}
	\rowcolor{red!15}{1,5}
	\Body
	u_{33} & u_{31} & u_{32} & u_{33} & u_{31} \\
	u_{13} & u_{11} & u_{12} & u_{13} & u_{11} \\
	u_{23} & u_{21} & u_{22} & u_{23} & u_{21} \\
	u_{33} & u_{31} & u_{32} & u_{33} & u_{31} \\
	u_{13} & u_{11} & u_{12} & u_{13} & u_{11}
\end{bNiceMatrix}.
\]
Regardless of the value of $p$, for $3\times 3$ \say{same} convolutions, one has to add two rows and two columns around the matrix $U$. The convolution operation $K*U$ defined by a $3\times 3$ filter $K$ is a linear map obtained by computing the Frobenius inner product of $K$ with all the contiguous $3\times 3$ submatrices of the padded input $U_P$. We show this procedure in the following example:
\[U_P=\begin{bNiceMatrix}[margin]
	\Block[draw,fill=blue!15,rounded-corners]{3-3}{} u_{33} & u_{31} & u_{32} & u_{33} & u_{31} \\
	u_{13} & u_{11} & u_{12} & u_{13} & u_{11} \\
	u_{23} & u_{21} & u_{22} & u_{23} & u_{21} \\
	u_{33} & u_{31} & u_{32} & u_{33} & u_{31} \\
	u_{13} & u_{11} & u_{12} & u_{13} & u_{11}
\end{bNiceMatrix},\quad  r_{11} = \mathrm{trace}\left(\begin{bmatrix}
	u_{33} & u_{31} & u_{32}  \\ 
	u_{13} & u_{11} & u_{12} \\ 
	u_{23} & u_{21} & u_{22}
\end{bmatrix}^TK\right),
\]
where $r_{11}$ is the first entry of the output $R$ obtained convolving $K$ with $U$. This operation is local as it only considers the entries of $3\times 3$ submatrices.

The same locality property holds for finite difference operators (see, e.g., \cite{thomas2013numerical,quarteroni2010numerical}), which are discrete approximations of the derivatives of a function given its nodal values sampled on a grid. For the PDE \eqref{eq:pde}, if the solution $u$ is regular enough in the spatial variables, there exists a $K\in\mathbb{R}^{3\times 3}$ such that
\[
(K*U)_{hk} = \mathcal{L}u\left(\bar{t},x_h,y_k\right) + \mathcal{O}\left(\dx^2\right),
\]
where $U\in\mathbb{R}^{p\times p}$ is defined as  $U_{hk}=u(\bar{t},x_h,y_k)$ for a fixed $\bar{t}\geq 0$. An explicit example is the well-known $5-$point formula (see \cite[Formula 25.3.30]{abramowitz1988handbook}) for approximating the Laplace operator $\Delta$, which is defined as
\[
\frac{1}{\dx^2}\left(\begin{bmatrix}
	0 & 1 & 0 \\
	1 & -4 & 1 \\
	0 & 1 & 0
\end{bmatrix} * U\right)_{hk} = \Delta u\left(\bar{t},x_h,y_k\right)\ + \mathcal{O}\left(\dx^2\right).
\] 
Similarly, any partial derivative of second-order or lower can be approximated to second-order accuracy with a $3\times 3$ convolution. Consequentially, one may observe that \eqref{eq:semi} with $L$, $D_i^a$, and $D_i^b$ realized by $3\times 3$ convolution operations, can provide a second-order accurate spatial discretization of the PDE \eqref{eq:pde}.

The convolution operator can be extended from matrices to higher-order tensors. Let $\boldsymbol{U}\in\mathbb{R}^{C_i\times p \times p}$ be a generic third-order tensor and $\fb{\mathcal{K}}\in\mathbb{R}^{C_o\times C_i\times K \times K}$ a fourth-order tensor representing the set of filters defining the convolution operation. We denote with $\boldsymbol{R}\in\mathbb{R}^{C_o\times p \times p}$ the result of the convolution operation $\fb{\mathcal{K}}*\boldsymbol{U}=\boldsymbol{R}$. The components of $\boldsymbol{R}$ can be characterized as
\[
\boldsymbol{R}_i = \sum_{j=1}^{C_i} \fb{\mathcal{K}}_{i,j}*\boldsymbol{U}_j\in\mathbb{R}^{p\times p},\quad i=1,...,C_o.
\]
As in the PyTorch library \cite{NEURIPS2019_9015}, we adopt the convention $\fb{\mathcal{K}}\in\mathbb{R}^{C_o \times C_i \times K\times K}$ where $C_i$ and $C_o$ are the numbers of input and output channels respectively, while the convolutional filters $\fb{\mathcal{K}}_{i,j}$ are of shape $K\times K$. 

\subsection{Error analysis for \texorpdfstring{$F_{\theta}$}{TEXT} based on convolution operations}
Based on the connections between finite differences and discrete convolution, we now show that building $F_{\theta}$ as
\begin{equation}\label{eq:pp}
	F_{\theta}\left(U\right)=\mathcal{L}_2\sigma\left(\mathcal{L}_1\left(U\right)\right)
\end{equation}
with $\sigma$ a suitable activation function applied entry-wise, $\mathcal{L}_1(U)=\fb{\mathcal{K}}*U+b_1$, and $\mathcal{L}_2\left(\sigma\left(\mathcal{L}_1(U)\right)\right)=\fb{\mathcal{H}}*\sigma\left(\mathcal{L}_1(U)\right)+b_2$ allows to exactly represent the right-hand side $F$ for a second-order accurate semi-discretization of the PDE \eqref{eq:pde}, where $b_1,b_2$ are bias terms added to the convolved inputs. In this case, $\theta=(\fb{\mathcal{K}},\fb{\mathcal{H}},b_1,b_2)$ represents the set of parameters defining $F_{\theta}$.

A common choice for $\sigma$ is the rectified linear unit $\sigma_1(x)=\mathrm{ReLU}(x):=\max\{0,x\}$. Some publications also considered powers of $\sigma_1$ (see, e.g., \cite{so2021searching,klusowski2018approximation}), calling them Rectified Power Units (RePUs). We first present a theoretical derivation based on $\sigma_2(x)=\mathrm{ReLU}^2(x)$. Then, we demonstrate a simplification of this result for linear partial differential equations (PDEs) based on $\sigma_1(x)=\mathrm{ReLU}(x)$. These two activation functions satisfy the important property
\begin{equation}\label{eq:polyProp}
	x^q = \mathrm{ReLU}^q\left(x\right) + \left(-1\right)^q \mathrm{ReLU}^q\left(-x\right),\,\,q\in\mathbb{N}.
\end{equation}
This identity allows for polynomials of degrees $1$ and $2$ to be represented by composing suitable linear functions and the two activation functions $\sigma_1$ and $\sigma_2$. These two activation functions are not polynomials, which allows them to be included in networks that can approximate sufficiently regular functions as accurately as desired, see \cite{pinkus1999approximation}. Therefore, they are more appealing than polynomials when designing neural network architectures.

As an immediate consequence of \eqref{eq:polyProp}, one can derive the following identities
\begin{subequations}
	\label{eq:identities}
	\begin{align}
		D*U &= \sigma_1\left(D*U\right)-\sigma_1\left(-D*U\right),\label{eq:reluLin}\\
		D*U &= \frac{1}{2}\left(\left(D*U+1\right)^2 - \left(D*U\right)^2\right) - \frac{1}{2} \label{eq:relu2Lin}\\
		&= \frac{1}{2}\left(\sigma_2\left(D*U+1\right)+\sigma_2\left(-D*U-1\right)\right.\nonumber\\
		&\left.-\sigma_2\left(D*U\right)-\sigma_2\left(-D*U\right)\right)-\frac{1}{2}, \nonumber\\
		\left(D_1*U\right)\odot \left(D_2*U\right) &= \frac{1}{2}\left(\left(\left(D_1+D_2\right)*U\right)^2-\left(D_1*U\right)^2-\left(D_2*U\right)^2\right)\label{eq:relu2Quad}\\
		&=\frac{1}{2}\left(\sigma_2\left(\left(D_1+D_2\right)*U\right)+\sigma_2\left(-\left(D_1+D_2\right)*U\right) \right. \nonumber\\
		&\left.- \sigma_2\left(D_1*U\right)-\sigma_2\left(-D_1*U\right)\right.\nonumber\\
		&\left.-\sigma_2\left(D_2*U\right)-\sigma_2\left(-D_2*U\right)\right).\nonumber
	\end{align}
\end{subequations}
These identities show how one can handle the linear term and the quadratic non-linearities arising in \eqref{eq:semi} using parametrizations like those in \eqref{eq:pp}. More explicitly, \eqref{eq:reluLin} and \eqref{eq:relu2Lin} show how to handle the linear term $L*U$, and \eqref{eq:relu2Quad} the quadratic interactions, as formalized in the following theorem.
\begin{theorem}\label{thm:repr}
	Let $U\in\mathbb{R}^{p\times p}$ and
	\[
	F\left(U\right) = L*U + \sum_{i=1}^I \beta_i \left(D_{2i-1}*U\right)\odot \left(D_{2i}*U\right) \in\mathbb{R}^{p\times p},
	\]
	for $L,D_{1},D_{2},...,D_{2I}\in\mathbb{R}^{3\times 3}$. Further, let $F_\theta$ be the parametric map defined by
	\begin{equation}
		F_{\theta}\left(U\right)=\fb{\mathcal{H}}*\sigma_2\left(\fb{\mathcal{K}}*U+b_1\right)+b_2.
	\end{equation}
	Then, $F_{\theta}$ can represent $F$ for suitably chosen parameters 
	\[
	\fb{\mathcal{K}}\in\mathbb{R}^{4+6I\times 1\times 3\times 3},\,\,\fb{\mathcal{H}}\in\mathbb{R}^{1\times 4+6I\times 1 \times 1},\,\,b_1\in\mathbb{R}^{4+6I},\,\,b_2\in\mathbb{R}.
	\]
\end{theorem}
\begin{proof}
	The proof is constructive since we report the exact expression of a family of weights that achieves the desired goal. We only specify the parts of the convolutional filters that are non-zero, which follow from \eqref{eq:identities}. We first fix the bias terms as
	\[
	b_1 = \begin{bmatrix} 1 & -1 & 0  & 0 & 0 & \cdots & 0\end{bmatrix}
	\]
	and $b_2=-1/2$. For the first convolutional filter, we instead set
	\[
	\fb{\mathcal{K}}_{1,1} =  -\fb{\mathcal{K}}_{2,1} = \fb{\mathcal{K}}_{3,1} = -\fb{\mathcal{K}}_{4,1}= L
	\]
	taking care of the linear part and
	\[
	\fb{\mathcal{K}}_{4+6i-5,1} = -\fb{\mathcal{K}}_{4+6i-4,1} = D_{2i-1}+D_{2i},\,\,i=1,...,I,
	\]
	\[
	\fb{\mathcal{K}}_{4+6i-3,1} = -\fb{\mathcal{K}}_{4+6i-2,1} = D_{2i-1},\,\,i=1,...,I,
	\]
	\[
	\fb{\mathcal{K}}_{4+6i-1,1} = -\fb{\mathcal{K}}_{4+6i,1} = D_{2i},\,\,i=1,...,I,
	\]
	which allows us to deal with quadratic interactions.
	This choice lets us get
	\begin{align*}
		\fb{\mathcal{K}}*U+b_1 &= \Big[L*U+1,-L*U-1,L*U,-L*U,\\
		&\left(D_1+D_2\right)*U,-\left(D_1+D_2\right)*U,D_{1}*U,-D_1*U,\\
		&D_2*U,-D_2*U,...,\left(D_{2I-1}+D_{2I}\right)*U,-\left(D_{2I-1}+D_{2I}\right)*U,\\
		&D_{2I-1}*U,-D_{2I-1}*U,D_{2I}*U,-D_{2I}*U\Big].
	\end{align*}
	Based again on \eqref{eq:identities}, we can conclude that, by setting
	\[
	\fb{\mathcal{H}}_{1,1} = \fb{\mathcal{H}}_{1,2} = - \fb{\mathcal{H}}_{1,3} = -\fb{\mathcal{H}}_{1,4} = \frac{1}{2},
	\]
	\begin{align*}
		\fb{\mathcal{H}}_{1,4+6i-5} &= \fb{\mathcal{H}}_{1,4+6i-4} = - \fb{\mathcal{H}}_{1,4+6i-3} \\
		&=-\fb{\mathcal{H}}_{1,4+6i-2} =-\fb{\mathcal{H}}_{1,4+6i-1} = -\fb{\mathcal{H}}_{1,4+6i} =\frac{\beta_i}{2},\,\,i=1,...,I,\end{align*}
	the result follows.
\end{proof}

We remark that the considered $F_{\theta}$ is not limited to representing only functions with the same structure as $F$, and this is the primary motivation behind the choice of not explicitly defining $F_{\theta}$ as 
\[
F_{\theta}\left(U\right) = L*U + \sum_{i=1}^I\beta_i \left(D_{2i-1}*U\right)\odot \left(D_{2i}*U\right).
\]
Indeed, it is generally hard to know if some temporal observations come from the discretization of a PDE. For this reason, we work with a more general neural network architecture. We note that any space of parametric functions $\mathcal{F}$ that contains parametric functions as those in theorem \ref{thm:repr} can represent $F$. That is to say that many overparametrized networks can be used while maintaining the theoretical guarantees.

We now simplify the construction for the case of linear PDEs.
\begin{theorem}\label{thm:reprLinear}
	Let $U\in\mathbb{R}^{p\times p}$ and
	\[
	F\left(U\right) = L*U\in\mathbb{R}^{p\times p},
	\]
	for $L\in\mathbb{R}^{3\times 3}$. Further, let $F_\theta$ be the parametric map defined by
	\begin{equation}
		F_{\theta}\left(U\right)=\fb{\mathcal{H}}*\sigma_1\left(\fb{\mathcal{K}}*U+b_1\right)+b_2.
	\end{equation}
	Then, $F_{\theta}$ can represent $F$ for suitably chosen parameters 
	\[
	\fb{\mathcal{K}}\in\mathbb{R}^{2\times 1\times 3\times 3},\,\,\fb{\mathcal{H}}\in\mathbb{R}^{1\times 2\times 1 \times 1},\,\,b_1\in\mathbb{R}^{2},\,\,b_2\in\mathbb{R}.
	\]
\end{theorem}
\begin{proof}
	One can set $b_1 = \left[0,0\right]$, $b_2=0$, $\fb{\mathcal{K}}_{1,1} = -\fb{\mathcal{K}}_{2,1} = L$, and $\fb{\mathcal{H}}_{1,2} = -\fb{\mathcal{H}}_{1,2} = 1$, which allows to conclude the proof.
\end{proof}

These two theorems can be extended to any activation function satisfying \eqref{eq:polyProp}. More explicitly, what is essential is to be able to represent linear maps and quadratic interactions composing $\sigma$ with suitable linear maps. Apart from polynomial activation functions, the LeakyReLU activation function also allows the representation of the right-hand side for linear PDEs. Indeed, such activation function is defined as $\mathrm{LeakyReLU}\left(x;a\right)=\max\left\{ax,x\right\},\,\,a\in \left(0,1\right)$, and theorem \ref{thm:reprLinear} extends to this activation function since
\[
x = \frac{1}{1+a}\left(\mathrm{LeakyReLU}\left(x;a\right)-\mathrm{LeakyReLU}\left(-x;a\right)\right).
\]

We conclude with a corollary combining this section's derivations with those in section \ref{se:pdeSols}, providing the central insight into our theoretical analysis.
\begin{corollary}\label{co:finalRes}
	Let $\Psi^{\dt}$ be a numerical method of order $r$. Let $U^0\mapsto\Phi\left(U^0\right)$ be the target one-step map obtained on a uniform mesh of $\Omega=[0,1]^2$ and with time step $\dt$, for the PDE \eqref{eq:pde}. 
	Further, assume the measurement error is either zero or of order equal to or higher than $r$ in time and $2$ in space.
	Then, the map $U^{0} \mapsto \Psi^{\dt}_{F_{\theta}}\left(U^0\right)$ can provide an approximation of $\Phi$ accurate to order $2$ in space and $r$ in time if $F_{\theta}$ is defined as in theorem \ref{thm:repr}.
\end{corollary}
We remark that the same result holds for linear PDEs and parametric spaces of functions as in theorem \ref{thm:reprLinear}.

\begin{proof}
	The proof immediately follows by combining the error splitting in \eqref{eq:errorSplit} and theorem \ref{thm:repr}. Indeed, one can get
	\[
	\begin{split}
		\left\|\Phi(U^0) - \Psi^{\dt}_{F_{\theta}}\left(U^{0}\right)\right\| &\leq \mathcal{O}\left(\dx^2\right) + \mathcal{O}\left(\dt^{r+1}\right) + \left\| \varepsilon^{1} \right\| + \left\|\Psi^{\dt}_{F}\left(U^{0}\right) - \Psi^{\dt}_{F_{\theta}}\left(U^{0}\right)\right\| \\
		&=\mathcal{O}\left(\dx^2\right) + \mathcal{O}\left(\dt^{r+1}\right)+ \left\|\varepsilon^{1}\right\|
	\end{split}
	\]
	since there exists a $\theta\in\mathcal{P}$ such that $\mathcal{F}\ni F_{\theta}=F$.
\end{proof}

This result ensures the possibility of getting approximations of PDE solutions that are second-order accurate in space and $r$ in time, with networks having a number of parameters growing linearly with the quadratic interactions in \eqref{eq:pde}.

\begin{remark}
	Our expressivity results rely on showing that the neural networks we consider can exactly represent a class of classical numerical methods and, hence, inherit their approximation rates. This reasoning is common when analyzing neural networks, e.g., \cite{adcock2022near,jin2020sympnets}. In the experiments, the parametric function $F_{\theta}$ is never forced to reproduce the spatial semi-discretization of a PDE, meaning that the presented derivations are intended as an upper bound for how poor the found approximation can be. If there is a better approximation of $\Phi$ that can be provided given the available data, the training phase will aim at that target.
\end{remark}
\section{Improving the stability of predictions}\label{se:stability}
The stability of an iterative method is often as crucial as its quantitative accuracy. When predicting the next frame of a time sequence, we consider a map stable if it is not overly sensitive to input perturbations. In our setting, iteratively applying the network leads to artifacts in the predictions, resulting in data points on which the network has not been trained. The goal is to minimize the impact of these artifacts, preventing a significant degradation in the accuracy of subsequent predictions. 

{We work on improving the stability of our networks on two levels. The spatial stability of the model is linked to how the parametric vector field $F_{\theta}$, which provides an approximate spatial semi-discretization of the PDE, processes the input matrices. For this, we have increased the size of the convolutional filters from $3\times 3$ to $5\times 5$ so that the pixel perturbations are better averaged out due to the wider window of action of the convolutional layers. For the temporal stability, we explore two strategies, both detailed in this section.} The first is noise injection, and the second involves building norm-preserving neural networks when dealing with data coming from norm-preserving PDEs. In the experiments associated with linear advection, which is known to be norm preserving, we will show that {preserving the norm results in the most significant stability improvements, followed by the noise injection strategy. Both these approaches improve on the results obtained without applying any temporal stability-enhancing strategy.} 
\subsection{Noise injection}
A technique often used in the literature to improve the training of a neural network is to introduce noise into the training set before sending the data through the network. This is generally additive noise with an expected value of zero. This helps to reduce the chance of overfitting the dataset. It has been shown, for example in \cite{bishop1995training}, that introducing noise in the inputs is equivalent to regularizing the network weights. Weight regularization is a commonly used technique in machine learning to improve the generalization capabilities of these parametric models, as discussed in \cite[Chapter 7]{Goodfellow-et-al-2016}.

To introduce noise, we modify the loss function to the form
\begin{equation}\label{eq:lossNoise}
	\mathcal{L}_{\varepsilon}\left(\theta,Q\right) = \frac{1}{N\cdot Q} \sum_{n=1}^N \sum_{q=1}^Q \left\|\mathcal{N}_{\theta}^q\left({U}_n^0+\delta_{n}\right) - \Phi^q\left(U_n^0\right)\right\|^2,
\end{equation}
where $\delta_n\sim\mathcal{U}(-\varepsilon,\varepsilon)^{p\times p}$ are independent identically distributed uniform random variables. More precisely, a new perturbation $\delta_n$ is generated at each training iteration. In the context of approximating the dynamics of unknown PDEs, one can also think of this noise injection strategy as a way to reduce the sensitivity of the learned dynamical system to perturbations in the initial condition. Indeed, \eqref{eq:lossNoise} ensures that the trajectories of a neighborhood of the initial condition $U_n^0$ are pushed towards the trajectory of $U_n^0$. 

\subsection{Norm preservation}\label{se:vecFieldCorrection}
We now move to the next technique we consider: incorporating in the neural network architecture a conservation law that the PDE is known to have. We analyze the linear advection equation $\partial_t u = {b}\cdot \nabla u$, $\nabla\cdot b=0$, with periodic boundary conditions on $\Omega=[0,1]^2$. For this PDE, the $L^2(\Omega)$ norm of $u$ is preserved since
\begin{align*}
	\frac{d}{dt}\frac{1}{2}\int_{\Omega}u^2\di{x}\di{y} &= \int_{\Omega} u\partial_tu\di{x}\di{y} = \int_{\Omega} u\left({b}\cdot\nabla u\right)\di{x}\di{y} \\
	&= \int_{\Omega}{b}\cdot \nabla\left(\frac{u^2}{2}\right)\di{x}\di{y} =\int_{\partial \Omega}{n}\cdot\left({b}\frac{u^2}{2}\right)\di{s} = 0,
\end{align*}
where ${n}$ is the outer pointing normal vector. To generate the training data, we use a norm-preserving numerical method, see appendix \ref{se:datagen} for details about it. 

To design a neural network that preserves the Frobenius norm of the input matrix, we first define a parametric space $\mathcal{F}$ of vector fields whose solutions preserve the Frobenius norm and then use a norm-preserving numerical method $\Psi^{\delta t}$. We define $\mathcal{F}$ as
\[
\mathcal{F} = \left\{F_{\theta}\left(U\right)=\mathcal{L}_2\sigma\left(\mathcal{L}_1\left(U\right)\right) - U\frac{\mathrm{trace}\left(U^T\mathcal{L}_2\sigma\left(\mathcal{L}_1\left(U\right)\right)\right)}{\mathrm{trace}\left(U^TU\right)}:\,\,\theta\in\mathcal{P}\right\}.
\]
$F_{\theta}(U)$ is the orthogonal projection of $\mathcal{L}_2\sigma(\mathcal{L}_1(U))$ onto the tangent space at $U$ of the manifold of $p\times p$ matrices having the same Frobenius norm of $U$. As a consequence, for every $U_0\in\mathbb{R}^{p\times p}$ and $F_{\theta}\in\mathcal{F}$, one has
\[
\frac{d}{dt}\left\|\Phi_{F_{\theta}}^t\left(U_0\right)\right\|^2 = 2\cdot\mathrm{trace}\left(\Phi_{F_{\theta}}^t\left(U_0\right)^TF_{\theta}\left(\Phi_{F_{\theta}}^t\left(U_0\right)\right)\right)=0,
\]
and hence $\left\|\Phi_{F_{\theta}}^t\left(U_0\right)\right\| = \left\|U_0\right\|$ for every $t\geq 0$.

We now present the numerical method $\Psi^{\delta t}$ used in the experiments. Such an integrator is a correction of the explicit Euler method aiming to preserve the norm through a Lagrange multiplier. The method is described as follows
\begin{align*}
	\tilde{U}^{m+1} &= U^m + \delta t F_{\theta}\left(U^m\right)\\
	{U}^{m+1} &= \tilde{U}^{m+1} + \lambda \tilde{U}^{m+1} =:\Psi^{\delta t}_{F_{\theta}}\left(U^m\right),
\end{align*}
where $\lambda\in\mathbb{R}$ is chosen so that $\|{U}^{m+1}\|^2 = \|U^m\|^2$. For this relatively simple constraint, $\lambda$ can be exactly computed as follows
\[
\left\|{U}^{m+1}\right\|^2 = \left(1+\lambda\right)^2\left\|\tilde{U}^{m+1}\right\|^2 = \left\|U^m\right\|^2\implies \lambda = -1\pm\frac{\left\|U^m\right\|}{\left\|\tilde{U}^{m+1}\right\|}.
\]
Given that when $\delta t=0$ one wants to have $\lambda=0$, the physical choice for $\lambda$ is the one with the plus sign, hence leading to
\begin{equation}\label{eq:projectionMethod}
	{U}^{m+1}=\Psi^{\delta t}_{F_{\theta}}\left(U^m\right) = \frac{U^m + \delta t F_{\theta}\left(U^m\right)}{\left\|U^m + \delta t F_{\theta}\left(U^m\right)\right\|}\left\|U^m\right\|.
\end{equation}

\section{Numerical experiments}\label{se:numerical}
This section collects numerical experiments supporting the network architecture introduced in section \ref{se:error}. All neural networks are implemented with the PyTorch library \cite{NEURIPS2019_9015} and are trained with the Adam optimizer. Our implementation can be found in \cite{self:code}. We consider the three following problems:
\begin{enumerate}
	\item linear advection equation, $\partial_t u = {b}\cdot \nabla u = \partial_x u + \partial_y u$,
	\item heat equation, $\partial_t u = \alpha \Delta u = \alpha \left(\partial_{xx}u+\partial_{yy}u\right)$, and
	\item Fisher equation, $\partial_t u = \alpha \Delta u + u(1-u)$.
\end{enumerate}
All these PDEs are considered with doubly periodic spatial boundary conditions and solved on $\Omega=[0,1]^2\subset\mathbb{R}^2$. We have not used the finite difference method to generate the training data to reduce the bias introduced by our data-generation technique. Indeed, we obtain the space-time observations from finite element simulations as described in appendix \ref{se:datagen}. These simulations yield a local truncation error of $\mathcal{O}(\dt^3 + \dx^2)$, allowing us to quantify our measurement error $\|\varepsilon^m\|$ within this section. Further, when we refer to a numerical method $\Psi^{\delta t}$, we perform $5$ sub-steps of step size $\delta t/5$ and omit this to simplify the notation.

In subsection \ref{sec:linearAdv}, we consider linear advection and compare the results obtained with two networks. These have the same number of parameters, but one is corrected for norm preservation, as presented in section \ref{se:vecFieldCorrection}, and based on the Lagrange multiplier method presented there, while the other is without these correction and projection steps. As with the other PDEs, we also compare the effects of noise injection on the error accumulation of the learned models.

In subsection \ref{sec:heat}, we deal with the heat equation with a neural network based on the presented theoretical derivations and choose the explicit Euler method as $\Psi^{\delta t}$. 
Finally, in subsection \ref{sec:fisher}, we report results for the Fisher equation based on the explicit Euler method.

To remain consistent with the results in section \ref{se:error}, we conduct numerical experiments using networks with the same number of channels as those we theoretically studied. However, based on numerical evidence, we employ $5\times 5$ convolutional filters rather than $3\times 3$ and $1\times 1$. The expressivity results still apply for these filters since they can represent convolutions with smaller filters, and, in practice, this choice leads to improved temporal stability of the network as a next-frame predictor. 

Two out of the three PDEs we consider are linear. Hence, we adopt the more efficient parametrization provided by theorem \ref{thm:reprLinear}, i.e., for the linear advection and heat equations, we define
\begin{equation}\label{eq:refArch}
	F_{\theta}\left(U\right)=\fb{\mathcal{H}}*\mathrm{ReLU}\left(\fb{\mathcal{K}}*U+b_1\right)+b_2
\end{equation}
where $\fb{\mathcal{K}}\in\mathbb{R}^{2\times 1\times 5\times 5}$, $\fb{\mathcal{H}}\in\mathbb{R}^{1\times 2\times 5\times 5}$, $b_1\in\mathbb{R}^2$ and $b_2\in\mathbb{R}$. For the Fisher equation, we use the parametrization based on $\mathrm{ReLU}^2(x)$ as in theorem \ref{thm:repr}.

To demonstrate the network's accuracy, we present figures showing the evolution of three metrics as the network makes predictions over $40$ time steps. These metrics rely on $30$ test initial conditions and are defined as
\begin{subequations}
	\begin{align}
		\texttt{maxE}\left(m\right) &= \max\left\{\left|\left(\mathcal{N}_{\theta}^m\left(U_n^0\right)-U_n^m\right)_{hk}\right|:n=1,...,30,h,k\in\left\{1,...,p\right\}\right\},\label{eq:maxE}\\
		\texttt{mse}\left(m\right) &= \frac{1}{30}\sum_{n=1}^{30} \left(\frac{1}{p^2}\left\|\mathcal{N}_{\theta}^m\left(U_n^0\right)-U_n^m\right\|^2\right)\label{eq:mseE},\\
		\texttt{rE}\left(m\right) &= \frac{1}{30}\sum_{n=1}^{30} \left(\frac{\left\|\mathcal{N}_{\theta}^m(U_n^0)-U_n^m\right\|}{\left\|U_n^m\right\|}\right)\label{eq:relE},\\
		\mathcal{N}_{\theta}^m&=\underbrace{\mathcal{N}_{\theta}\circ ... \circ \mathcal{N}_{\theta}}_{m\text{ times}},\quad m=1,...,40.\nonumber
	\end{align}
\end{subequations}
We refer to \eqref{eq:maxE} as the maximum absolute error, to \eqref{eq:mseE} as the mean squared error (MSE), and to \eqref{eq:relE} as the average relative error. All experiments are conducted with $p=100$, i.e., with matrices of size $100\times 100$.

To train the networks, we optimize either the function $\mathcal{L}(\theta,M)$ in \eqref{eq:lossPixel} or  $\mathcal{L}_{\varepsilon}(\theta,M)$ in \eqref{eq:lossNoise}. For noise injection, we set the noise magnitude to $\varepsilon=10^{-2}$. We adopt a training procedure which, as described in algorithm \ref{alg:progr}, pre-trains the network on shorter sequences of snapshots, decreasing the learning rate as we increase the sequence length, as in curriculum learning \cite{soviany2022curriculum}. The algorithm is presented for the complete set of training initial conditions and uses a step learning rate scheduler, dividing the learning rate by $10$ every $135$ epochs. In practice, we implement a mini-batch version of this algorithm with a cyclic learning rate scheduler, \cite{smith2017cyclical}; however, the training procedure follows the same logic. All experiments use batches of size $32$, i.e., $32$ different initial conditions.
\begin{algorithm}[ht!]
	\caption{Training with the full training set and step learning rate scheduler}\label{alg:progr}
	\begin{algorithmic}[1]
		\State \text{Initialize $\mathcal{N}_{\theta}$}
		\State \texttt{Epochs $\gets 300$}
		\State $\ell \gets 5\cdot 10^{-3}$ 
		\Comment{Set the starting learning rate}
		\For{$M\in [2,3,4]$}
		\State $\texttt{lr} \gets \ell$
		\While{\texttt{e<Epochs}}
		\State \text{One optimization step of $\mathcal{L}(\theta,M)$ with learning rate $\texttt{lr}$}
		\If{$\texttt{e}\in [135,270]$}
		\State $\texttt{lr} \gets \texttt{lr}/10$ \Comment{Learning rate scheduler}
		\EndIf
		\State $\texttt{e}\gets \texttt{e}+1$
		\EndWhile
		\State $\ell \gets \ell/2$ \Comment{Now we add another time step, but we do smaller optimization steps}
		\EndFor
	\end{algorithmic}
\end{algorithm}
\subsection{Linear advection equation}\label{sec:linearAdv}
Before presenting the numerical results, we briefly recall the neural networks we consider. We define two neural networks: one corresponds to the explicit Euler method applied to the vector field $F_{\theta}$ in \eqref{eq:refArch}, and the other corresponds to the projected version presented in subsection \ref{se:vecFieldCorrection}.

\begin{figure}[t]
	\begin{subfigure}{.45\textwidth}
		\centering
		\includegraphics[width=\textwidth]{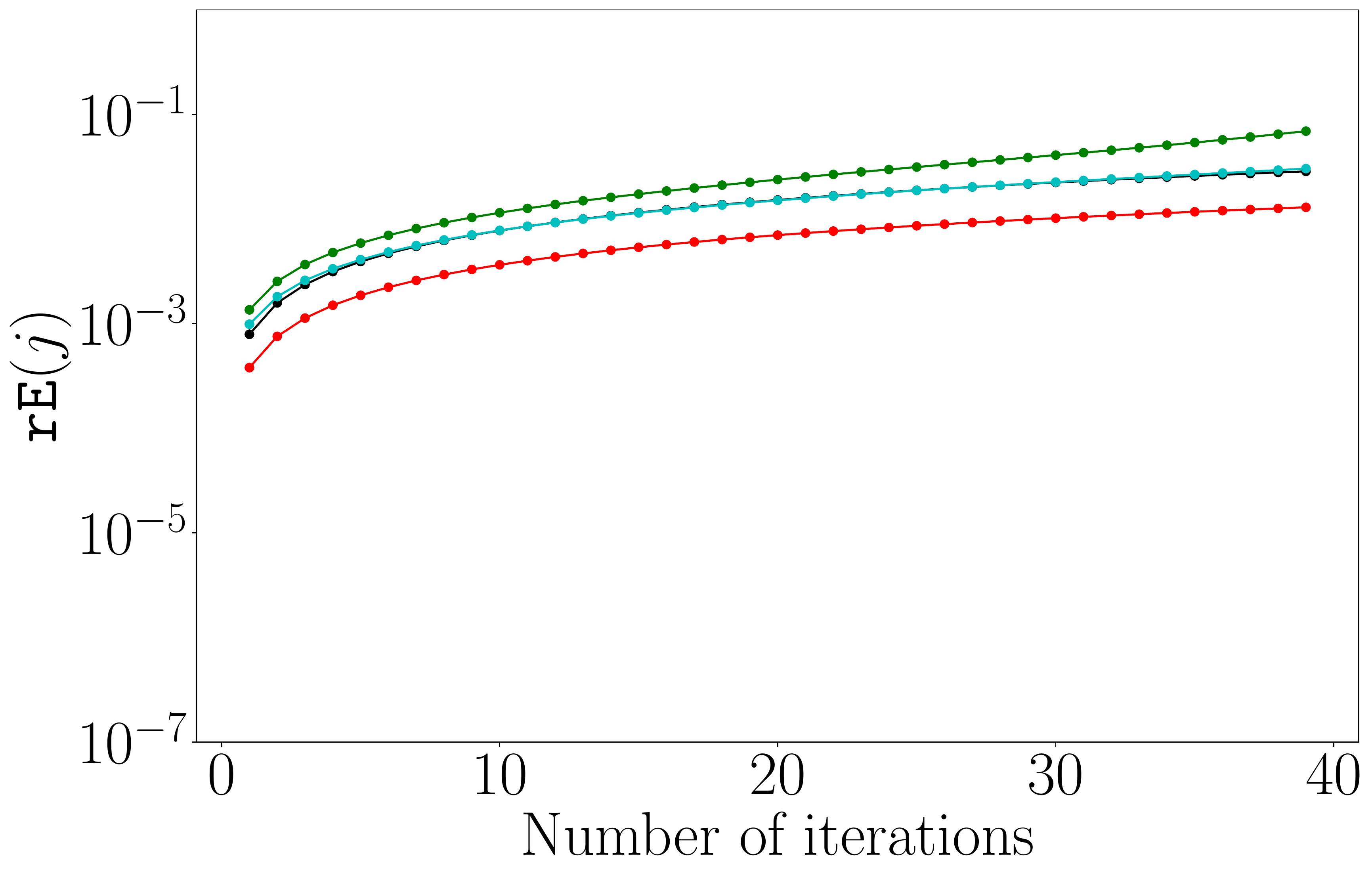}
		\caption{Average relative error over a batch of $30$ test initial conditions.}
		\label{fig:relativeAdv}
	\end{subfigure}
	\hfill
	\begin{subfigure}{.45\textwidth}
		\centering
		\includegraphics[width=\textwidth]{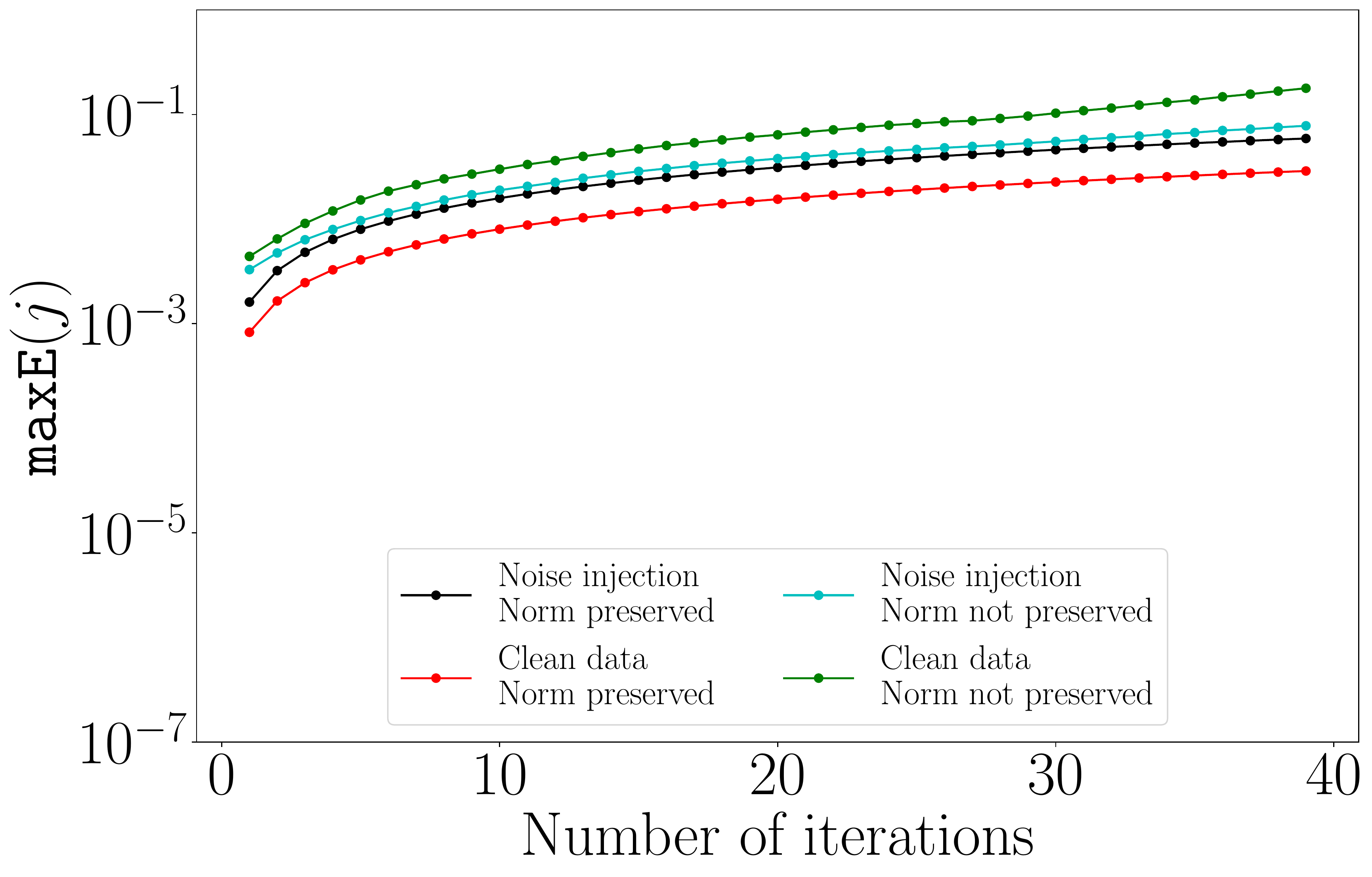}
		\caption{Maximum absolute error over a batch of $30$ test initial conditions.}
		\label{fig:maxAdv}
	\end{subfigure}
	
	\centering
	\begin{subfigure}{.45\textwidth}
		\centering
		\includegraphics[width=\textwidth]{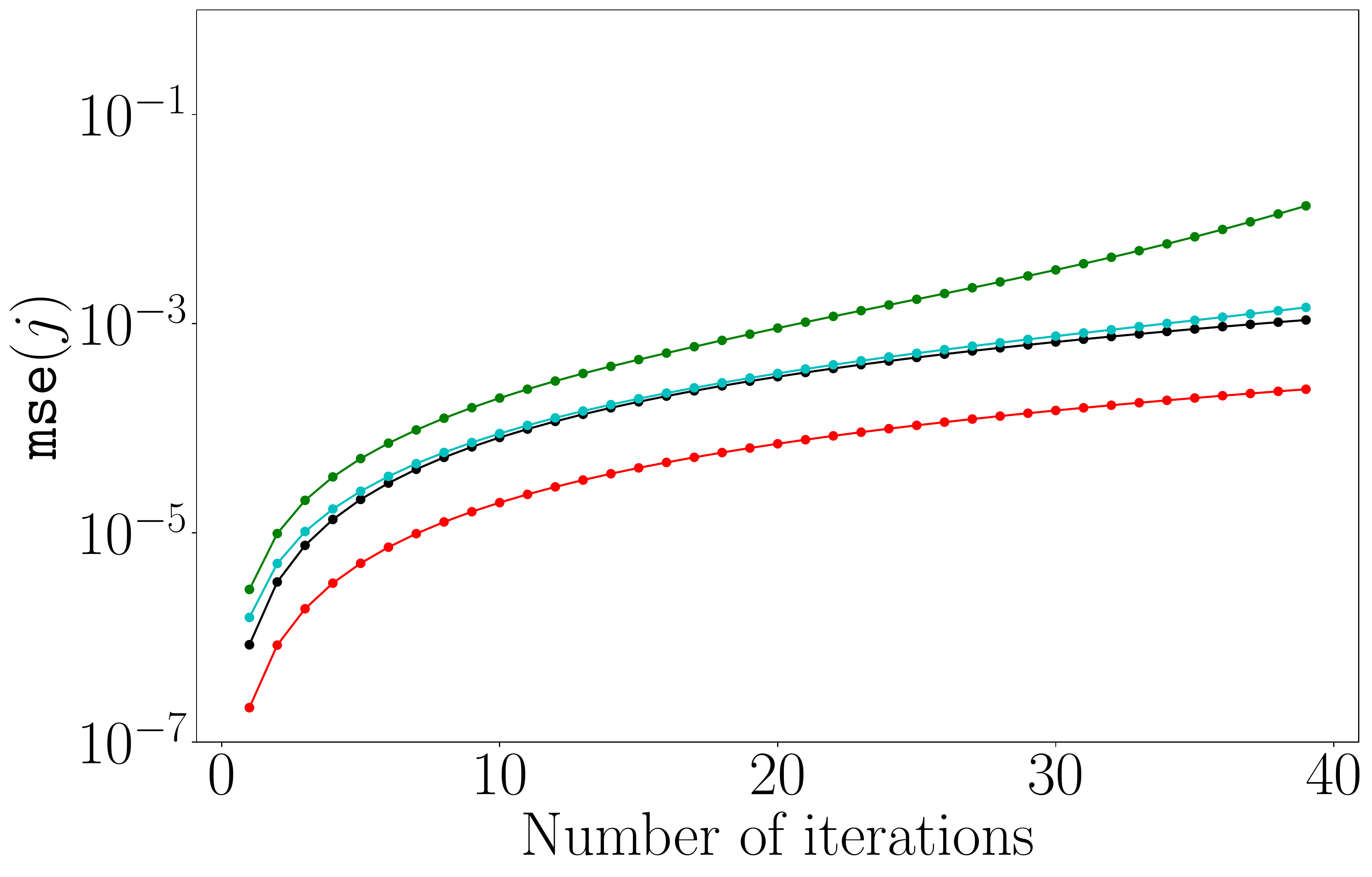}
		\caption{Average of MSE over a batch of $30$ test initial conditions.}
		\label{fig:mseAdv}
	\end{subfigure}
	\caption{Test errors for the linear advection equation.}
\label{fig:plotsAdv}
\end{figure}

In these experiments, we compare the effects of noise injection and norm preservation on the accuracy and stability of the time series approximations provided by the trained neural network. The results are presented in figure \ref{fig:plotsAdv}. We note that, as expected, even though the neural networks we consider have the same number of parameters, their different arrangements considerably change the results we recover. Firstly, injecting noise while training the network and preserving the norm of the initial condition both improve the stability of the predictions since the error accumulates at a lower speed than without these changes. Secondly, even if we carefully specify the correct value of the norm to preserve in the experiments, combining the two strategies does not improve the results beyond only injecting noise. To conclude, we highlight that these experiments imply that such small networks not only have the potential to be expressive enough to represent the desired target map $\Phi$, but that one can also find a set of weights leading to an accurate and stable solution.

\subsection{Heat equation}\label{sec:heat}
For the heat equation, we consider the same neural network architecture used for linear advection. The time step is imposed following the Courant--Friedrichs--Lewy (CFL) condition\footnote{The details on the values of the diffusivity constant $\alpha$ can be found in appendix \ref{sec:data:linadv}} and is
\[
\dt = 0.24\cdot \frac{\dx^2}{\alpha } \approx 2.445\cdot 10^{-3}.
\]
We report the results of this experiment in figure \ref{fig:plotsHeat}. The network is based on the explicit Euler method. Unlike what occurs for the linear advection equation, we notice that introducing additive noise in the training procedure worsens the results. This behavior is primarily a consequence of the inability of the training phase to find a good set of weights. While with clean data we can consistently reach a loss value of the order of $10^{-8}$, the additive injection of noise leads to a final training loss of the order of $10^{-5}$. The dynamics of the heat equation are dissipative, leading to the need for a more complex regularization strategy than additive noise.
\begin{figure}[t]
\begin{subfigure}{.45\textwidth}
	\centering
	\includegraphics[width=\textwidth]{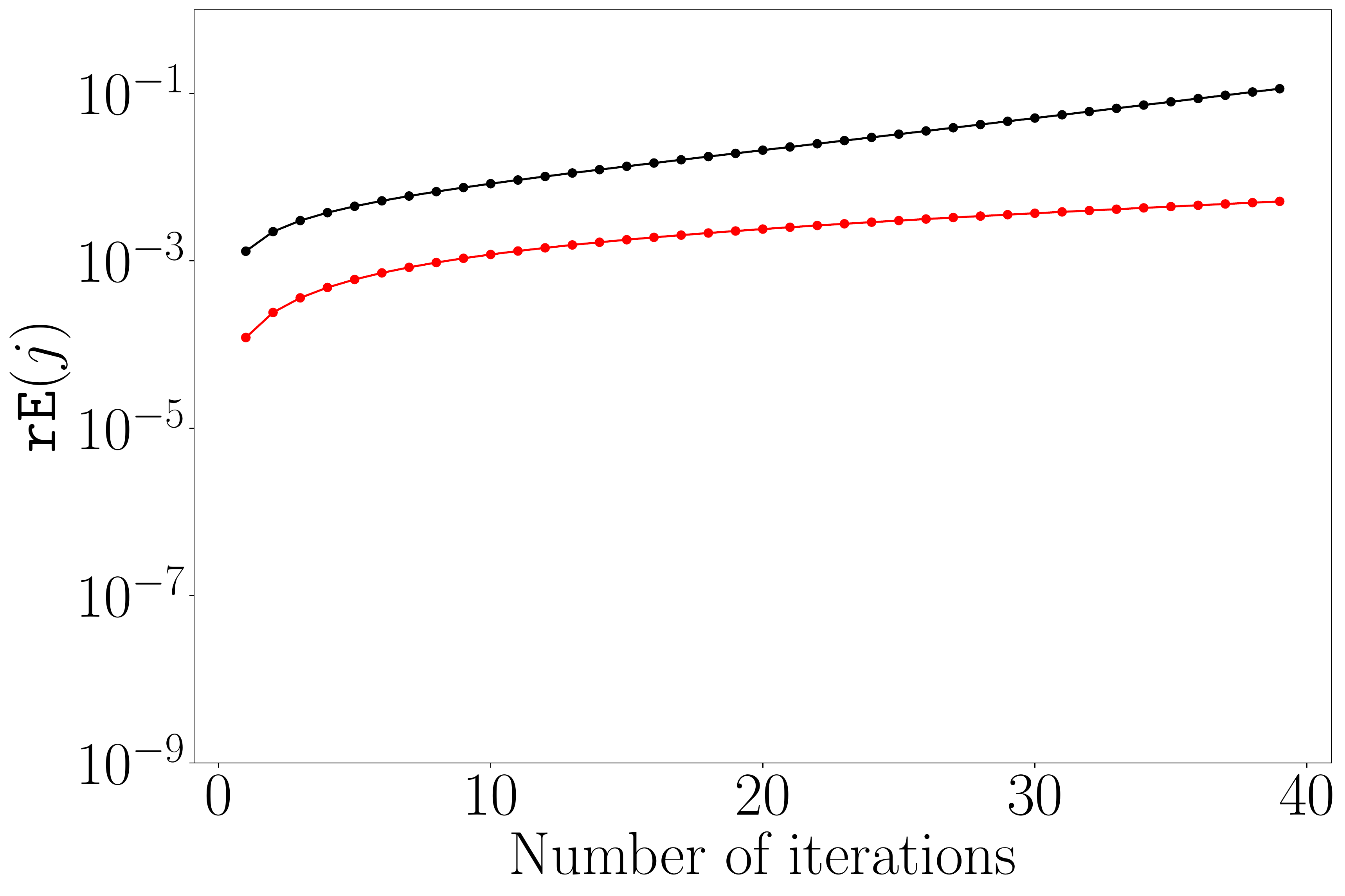}
	\caption{Average relative error over a batch of $30$ test initial conditions.}
	\label{fig:relativeHeat}
\end{subfigure}
\hfill
\begin{subfigure}{.45\textwidth}
	\centering
	\includegraphics[width=\textwidth]{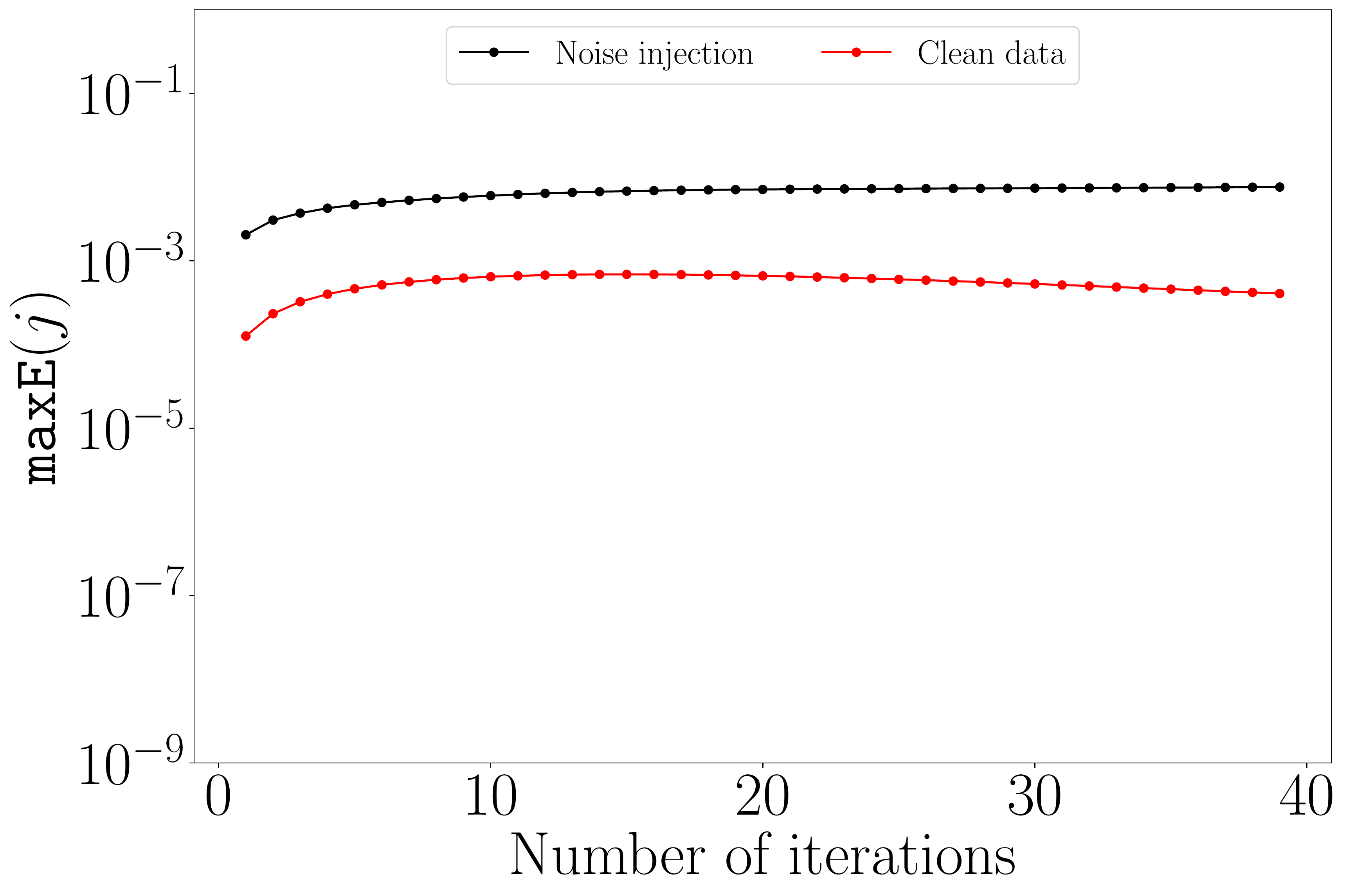}
	\caption{Maximum absolute error over a batch of $30$ test initial conditions.}
	\label{fig:maxHeat}
\end{subfigure}

\centering
\begin{subfigure}{.45\textwidth}
	\centering
	\includegraphics[width=\textwidth]{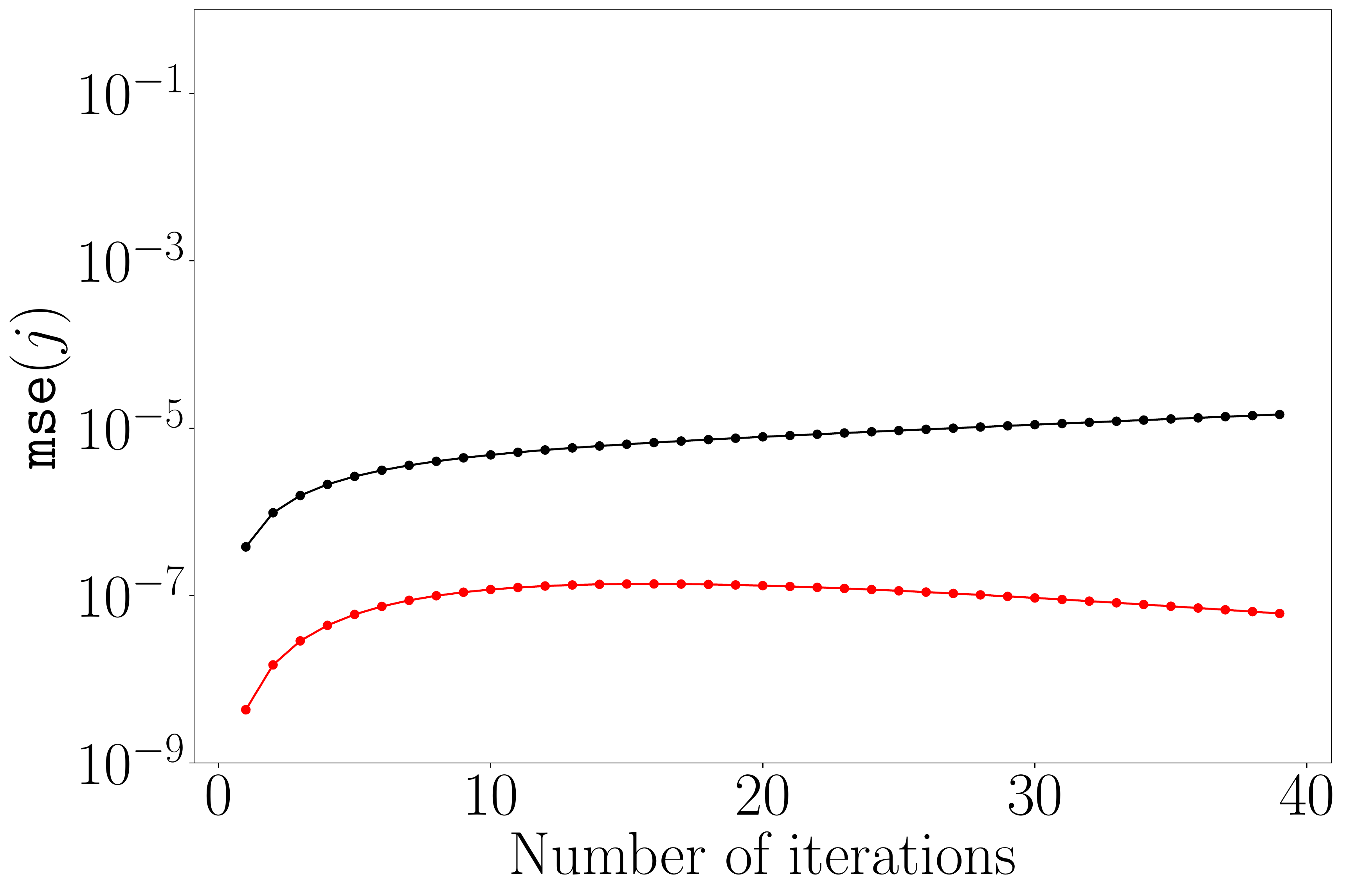}
	\caption{Average of MSE over a batch of $30$ test initial conditions.}
	\label{fig:mseHeat}
\end{subfigure}
\caption{Test errors for the heat equation.}
\label{fig:plotsHeat}
\end{figure}

\subsection{Fisher equation}\label{sec:fisher}
The Fisher equation is a nonlinear PDE with $I=1$ quadratic nonlinear interactions (see theorem \ref{thm:repr}). For this reason, the theoretical derivations in section \ref{se:error} guarantee that a CNN $F_{\theta}$ with two layers, $\mathrm{ReLU}^2$ as activation function, and ten channels is sufficient to represent the semi-discretization of the PDE corresponding to centered finite differences of the second order. This architecture is precisely the one we use in the experiments.

The dynamics of this system are more complicated than those of the heat equation. As presented in appendix \ref{sec:data:fisher}, we generate the initial conditions to train and test the network similarly to the heat equation, and the time step $\dt$ has the same value. To obtain the results in figure \ref{fig:plotsFisher}, we select only the training and test initial conditions for which $\|U^0\|_F>10$. We choose the value $10$ to obtain a more uniform dataset and avoid working with inputs of entirely different scales. This change is due to the low frequency of the initial conditions with small norms obtained through random data generation. Overall, the numerical results in figure \ref{fig:plotsFisher} align with those we get for the other two PDEs. As with the heat equation, the dissipative nature of this PDE makes noise injection less effective than it is for linear advection, and possibly different regularization strategies are needed to improve the stability of the network.
\begin{figure}[ht!]
\begin{subfigure}{.45\textwidth}
\centering
\includegraphics[width=\textwidth]{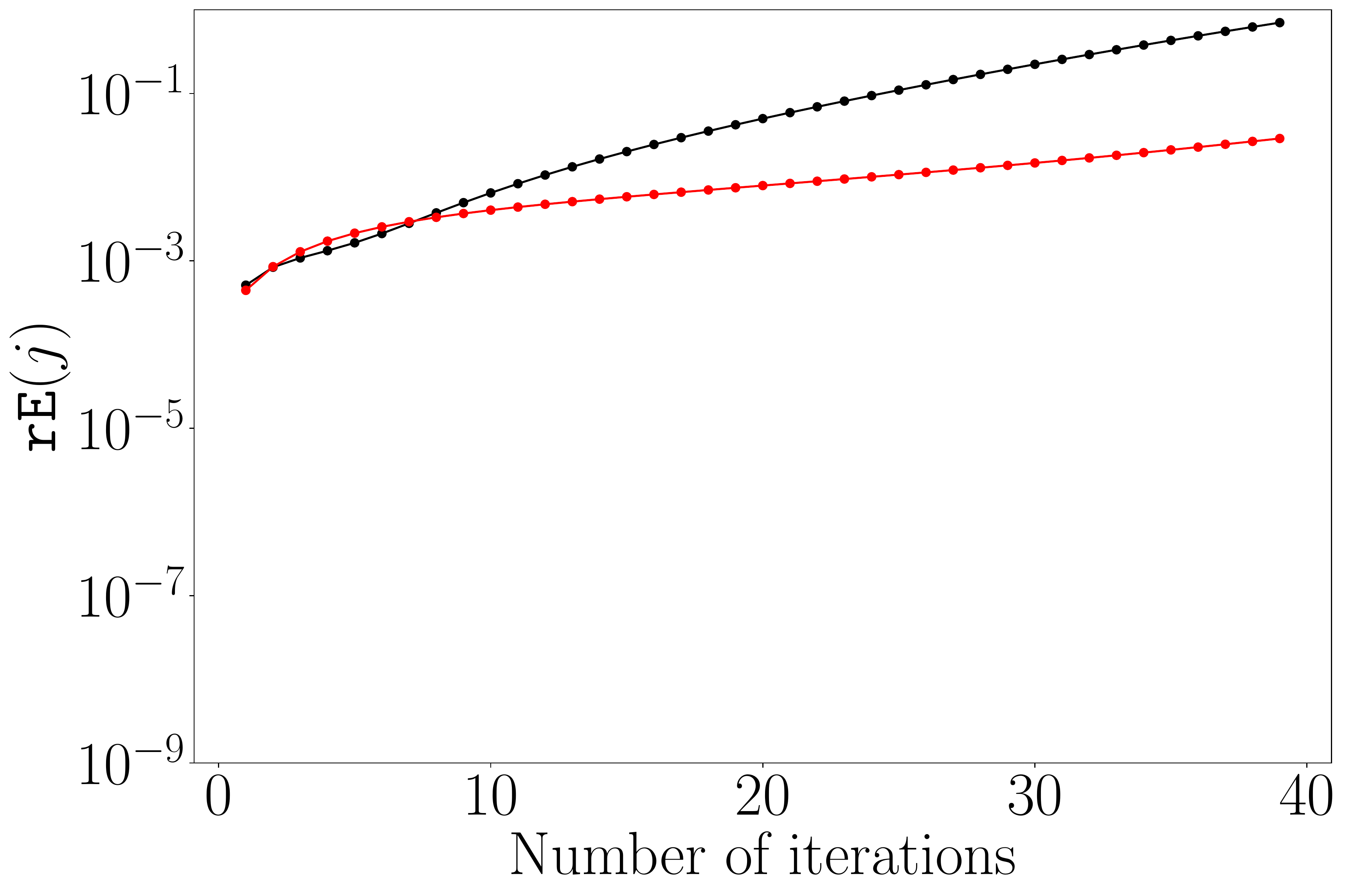}
\caption{Average relative error over a batch of $30$ test initial conditions.}
\label{fig:relativeFisher}
\end{subfigure}
\hfill
\begin{subfigure}{.45\textwidth}
\centering
\includegraphics[width=\textwidth]{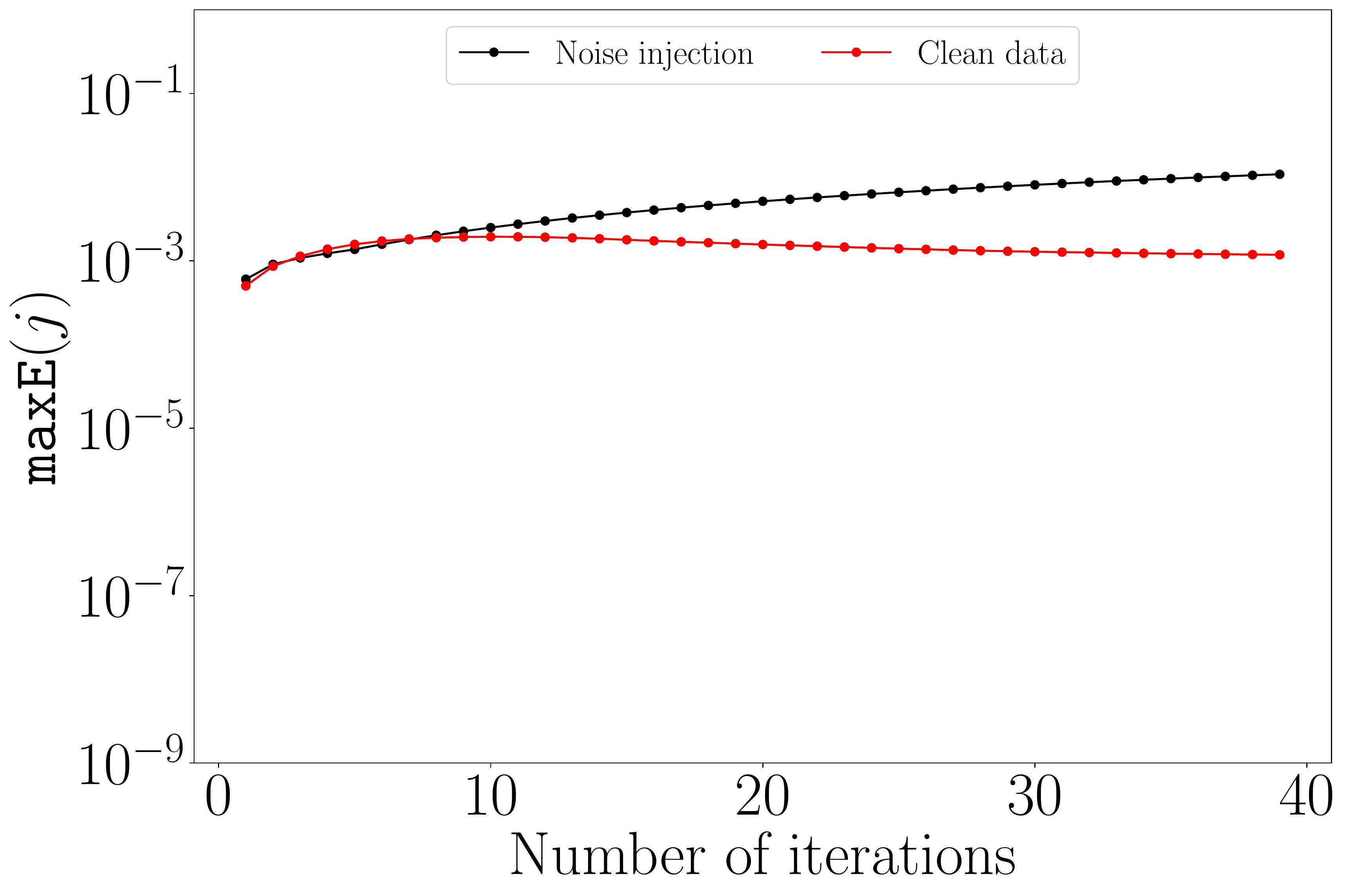}
\caption{Maximum absolute error over a batch of $30$ test initial conditions.}
\label{fig:maxFisher}
\end{subfigure}

\centering
\begin{subfigure}{.45\textwidth}
\centering
\includegraphics[width=\textwidth]{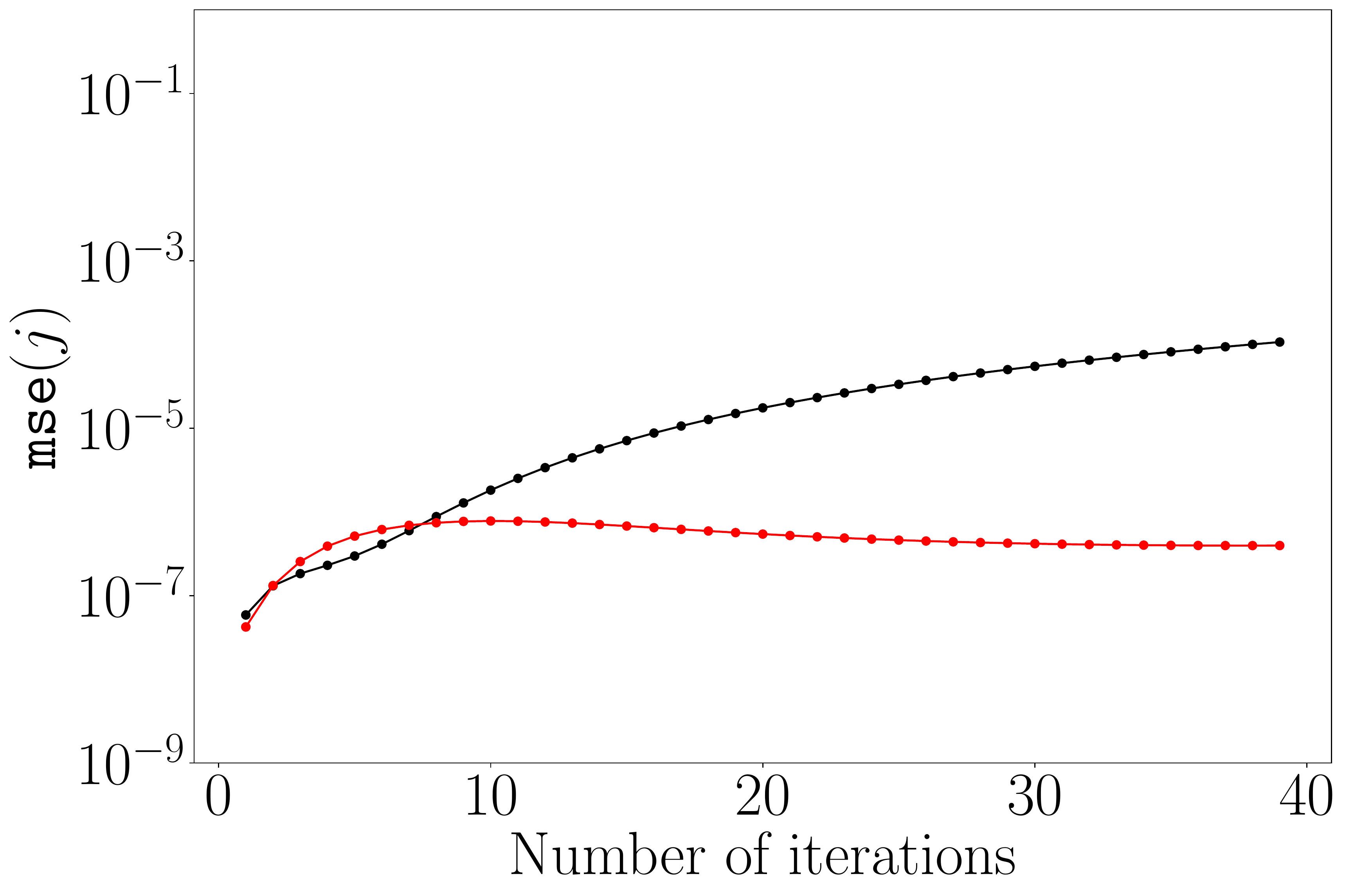}
\caption{Average of MSE over a batch of $30$ test initial conditions.}
\label{fig:mseFisher}
\end{subfigure}
\caption{Test errors for the Fisher equation.}
\label{fig:plotsFisher}
\end{figure}
\section{Conclusion and further work} \label{se:conc}

In this manuscript, we presented expressivity results for two-layer CNNs used for approximating temporal sequences. We focused on PDE space-time observations, examining CNNs' ability to represent PDE solutions for PDEs with quadratic nonlinear terms. 

Exploiting the connections between finite difference operators and discrete convolution, we showed that it is sufficient to consider relatively small two-layer networks, where the size increases linearly in the number of quadratic interactions in the unknown model. Moreover, the networks that we investigated can represent broader classes of sequence data than just PDE solutions.

We also experimentally analyzed the effects of norm preservation and noise injection. {Norm preservation as a means to improve the network stability can be beneficial for other PDE-generated datasets, such as for the Schrödinger equation, see \cite{sanz1984methods}. Furthermore, the effects of other conservation laws could also be considered. Our proposed approach to leveraging the conservation laws relies on projection methods, which extend naturally to other conserved properties. However, it would be interesting to understand if stability is enhanced by applying these other conservation properties and if the simplicity of projection methods becomes a limiting factor in the extension to more general conservation laws or if it remains a valuable strategy.}

While our theoretical results rely on specific choices of activation functions, in practice, many more activation functions are successfully included in neural networks for time-series approximations. Hence, it is of interest to extend our results to other classes of functions.

\section*{Acknowledgments} 
This work was partially supported by a grant from the Simons Foundation (DM), the ERCIM Alain Bensoussan Fellowship Programme and Marie Sk{\l}odowska-Curie grant agreement No 101108679 (JJ).
Additional funding was received from the European Union’s Horizon 2020 research and innovation program under the MSCA-ETN grant agreement No 860124 and MSCA-SE grant agreement 101131557, from the Research Council of Norway, and from the Trond Mohn Foundation (EC) and (BO).

\bibliography{biblio}

\appendix

\section{Data generation} \label{se:datagen}

We have relied on finite element discretizations to generate data throughout our numerical experiments in section \ref{se:numerical}. This dataset can be found in \cite{self:data}. For completeness, we now briefly outline the methods used. Throughout, spatially, we utilize piecewise continuous linear Lagrange finite elements which, to a reasonable extent, respect the PDE dynamics. Crucially, any generated data must be interpolated (preferably in a structure-preserving manner) as a set of matrices. Recall that the matrix entries represent points in two-dimensional space, with each matrix corresponding to a discrete \say{snapshot} of the solution in time. With this matrix structure in mind, we choose to mesh our domain with regular quadrilaterals (as opposed to the more typical triangulation) and utilize bi-quadratic basis functions. Indeed, by doing so with linear elements, the degrees of freedom of the methods will be represented precisely by the matrix values. In particular, we shall write this finite element space as $\fes$, which implicitly depends on the meshing of our domain. Throughout, we fix our matrix dimension to be $\mathbb{R}^{100 \times 100}$, which fixes our mesh resolution to be $p=100$, where $p$ corresponds to the number of subrectangles in both the $x$ and $y$ direction. Temporally, we utilize second-order implicit time-stepping methods and define the temporal evolution through the time step $\dt$. For ease of implementation, we use the finite element library Firedrake \cite{Firedrake}, and our implementation can be found in \cite{self:code}. Each simulation is initialized by \say{random} initial data, which is problem-dependent, as described below.

\subsection{Linear advection} \label{sec:data:linadv}

Let $u=u(t,x,y)$, where $(x,y) \in \Omega := [0,1]^2$ doubly periodic and $t \in [0,T]$. Then, we may express linear advection as
\begin{equation*}
\begin{cases}
\partial_t u = {b} \cdot \grad u,\\
u\left(0,x,y\right)=u_0\left(x,y\right)
\end{cases}
\end{equation*}
where ${b} \in \mathbb{R}^2$ is some specified constant. To clarify exposition, we shall focus on the first step of the method, which may be easily extrapolated at all times. Let $U_0$ be given by the interpolation of the initial data $u_0$ into the finite element space, then the solution at the next time step is given by seeking $U_1 \in \fes$ such that
\begin{equation*}
\int_{\Omega} \bc{ \frac{U_1-U_0}{\dt} + {b} \cdot \grad{U_{\frac12}} } \phi \di{x}\di{y}
\quad
\forall \phi \in \fes
,
\end{equation*}
where $U_{\frac12} = \frac12 \bc{U_0+U_1}$. 

One fundamental property of this method is that it preserves both a discrete mass and momentum. To be more concise, the mass
$
\int_{\Omega} U_1 \di{x}\di{y} = \int_{\Omega} U_0 \di{x}\di{y}
$
is constant over time, as can be observed by the periodic boundary conditions and after choosing $\phi = 1$. More importantly, the momentum
$
\int_{\Omega} U_1^2 \di{x}\di{y}
=
\int_{\Omega} U_0^2 \di{x}\di{y}
$
is conserved, as may be observed after choosing $\phi = U_{\frac12}$. Conservation of momentum is equivalent to preserving the norm of the underlying matrix.

In section \ref{sec:linearAdv}, our random initial conditions are generated by 
\begin{equation*}
u_0
=
\sin{2\pi \alpha_1 \left(x-x_s\right)} \cos{2\pi \alpha_2 \left(y-y_s\right)} + 1
,
\end{equation*}
where $\alpha_i \sim U(\{5,6,7,8\})$ are independently sampled with equal likelihood and $x_s, y_s \sim U([0,1])$ are sampled from a unitary uniform distribution.
Further, we fix the time step to be $\dt = 0.02$ and $\vec{b} = (1,1)$ throughout the experiments.

\subsection{Heat equation} \label{sec:data:heat}

Here, we utilize the same setup as the heat equation. That is to say, we let $u=u(t,x,y)$ where $(x,y) \in \Omega = [0,1]^2$ (doubly periodic) and $t \in [0,T]$, be given by
\begin{equation*}
\begin{cases}
u_t = \alpha \Delta u \\
u(0,x,y) = u_0(x,y)
,
\end{cases}
\end{equation*}
where $\alpha \in \mathbb{R}$ is the dissipation constant. Letting $U_0$ be the interpolation of some given randomized initial data into the finite element space $\fes$, the numerical method is described by seeking $U_1 \in \fes$ such that
\begin{equation*}
\int_\Omega \left(\bc{\frac{U_1-U_0}{\dt}} \phi
+ \alpha \grad U_{\frac12} \cdot \grad \phi\right)
\di{x}\di{y}
= 0
\quad
\forall \phi \in \fes
.
\end{equation*}
In section \ref{sec:heat}, we exploit the dissipative behavior of this equation. Indeed, this numerical method respects the physical rate of dissipation. By choosing $\phi = U^{\frac12}$, we observe
\begin{equation*}
\int_\Omega U_1^2 \di{x}\di{y}
=
\int_\Omega \left(U_0^2
- \alpha \grad U_{\frac12}^2\right)
\di{x}\di{y}
,
\end{equation*}
which is consistent with the true rate of dissipation in the heat equation. 

In section \ref{sec:heat}, our random initial conditions are generated by
\begin{equation} \label{eq:heat:ic}
u_0
=
\sin{k \pi \left(x-x_p\right)} \sin{k \pi \left(y-y_p\right)}
,
\end{equation}
where $k \sim U(\{2,3,4,5,6,7\})$ is randomly selected and $x_p, y_p \sim \mathcal{N}(1,0.5)$ are normally distributed with mean $1$ and variance $0.5$. Further, throughout our experiments, we fix $\dt = 0.024$ and $\alpha = 0.01$.

\subsection{Fisher equation} \label{sec:data:fisher}

By modifying the heat equation (maintaining our doubly periodic spatial domain), let us consider a nonlinear reaction-diffusion equation. In particular, let $u=u(t,x,y)$ solve the reaction-diffusion equation
\begin{equation} \label{eq:fisher}
\begin{cases}
\begin{alignedat}{2}
& u_t =  \alpha \Delta u + u\left(1-u\right) \\
& u(0,x,y) = u_0(x,y)
,
\end{alignedat}
\end{cases}
\end{equation}
where $u_0$ represents our randomized initial data. Letting $U_0$ be given by the interpolation of $u_0$ into the finite element space, the first step of the method is given by seeking $U_1 \in \fes$ such that
\begin{equation*}
\int_\Omega 
\bc{\frac{U_1-U_0}{\dt}} \phi
+ \alpha \grad U_{\frac12} \cdot \grad \phi
- \bc{ U_{\frac12} + \frac13 \bc{
U_1^2 + U_1 U_0 + U_0^2
}}
\phi \di{x}\di{y}
=
0,
\end{equation*}
for all $\phi \in \fes$. Note here that the choice of temporal discretization for the nonlinear term is not unique, and we have chosen a second-order accurate temporal discretization. 
In section \ref{sec:fisher}, the random initial conditions generated are the same as for the heat equation due to similarities in the setup (see \eqref{eq:heat:ic}).
\end{document}